\documentclass[11pt,reqno]{amsart}
\usepackage{amsfonts}
\usepackage{amsmath,amsthm,amssymb,url,listings,amsfonts,amscd}
\usepackage[alphabetic,initials]{amsrefs}
\usepackage{mathrsfs}
\usepackage{lipsum}
\usepackage{bbding}
\usepackage{graphicx,latexsym}
\usepackage{hyperref}
\usepackage{scrtime}
\usepackage{color}

\numberwithin{equation}{section}

\theoremstyle{plain}
\newtheorem{lemma}{Lemma}[section]
\newtheorem{theorem}[lemma]{Theorem}

\newtheorem{proposition}[lemma]{Proposition}

\theoremstyle{definition}

\newtheorem{remark}{Remark}

\renewcommand{\Re}{\operatorname{Re}}

\begin{document}
	
\title{Hybrid subconvexity bounds for twists of $\rm GL(3)$ $L$-functions}
		
\author{Xin Wang}
\address{School of Mathematics, Shandong University
                       \\Jinan, Shandong 250100, China}
\email{wangx2021@mail.sdu.edu.cn}

\date{}

\author{Tengyou Zhu}
\address{School of Mathematics, Shandong University
                       \\Jinan, Shandong 250100, China}
\email{zhuty@mail.sdu.edu.cn}	

\date{}

\begin{abstract}
Let $F$ be a  Hecke-Maass
cusp form on $\rm SL(3,\mathbb Z)$ and $\chi$ a primitive Dirichlet character
of prime power conductor $\mathfrak{q}=p^k$ with $p$ prime.
In this paper we will prove the following subconvexity bound
$$
L\left(\frac{1}{2}+it, F \times \chi\right)\ll_{\pi,\varepsilon}
p^{3/4}\big(\mathfrak{q}(1+|t|)\big)^{3/4-3/40+\varepsilon},
$$
for any $\varepsilon >0$ and $t \in \mathbb{R}$.
\end{abstract}

	
	\maketitle
\section{Introduction}\label{introduction}

Let $L(s,f)$ be a general $L$-function with an analytic conductor $\mathfrak{q}(f)$. By the functional equation and the Phragm\'{e}n-Lindel\"{o}f convexity principle, there is a convexity bound
$L(s, f)\ll \mathfrak{q}(f)^{1/4+\varepsilon}$ on the critical line
$\Re(s)=1/2$.
In many applications of number theory, it is necessary to beat convexity bound to get bounds of the form
$L(1/2+it,f)\ll\mathfrak{q}(f)^{1/4-\delta+\varepsilon}$ for some $\delta>0$. The famous Riemann hypothesis implies that $L(1/2+it,f)\ll\mathfrak{q}(f)^\varepsilon$ which is known as Lindel\"{o}f hypothesis.
In this paper, we get a new hybrid subconvexity bound for $\rm GL(3)$ $L$-functions twisted by a primitive Dirichlet character modulo $\mathfrak{q}=p^k$ with $p$ prime.

In the last two decades, people extend the results on $\rm GL(2)$ and $\rm GL(3)$
 $L$-functions to different aspects, either in $\mathfrak{q}$-aspect or in $t$-aspect.
(see\cite{Agg2}, \cite{BD}, \cite{Li}, \cite{Mun1},
\cite{Mun6}, \cite{LS} and so on).

Let $F$ be a Hecke-Maass cusp form for $\rm SL(3,\mathbb{Z})$
with the normalized Fourier coefficients $A(m,n)$. The $L$-function associated with $F$ is
$$
L(s, F)=\sum\limits_{n\geq1}\frac{A(1,n)}{n^{s}},\ \ \ \Re(s)>1.
$$
Let $\chi$ be a primitive Dirichlet character of  conductor $\mathfrak{q}$.
The  twisted $L$-function is defined by
$$
L(s, F \times \chi)=\sum\limits_{n\geq1}
\frac{A(1,n)\chi(n)}{n^{s}}, \ \ \ \Re(s)>1,
$$
which has analytic continuation to the whole complex plane. We consider the $L$-values at the point $1/2+it$.
The Phragm\'en-Lindel\"of principle implies the convexity bound
$$
L\left(\frac{1}{2}+it,F\times \chi\right)\ll \left(\mathfrak{q}(1+|t|)\right)^{3/4+\varepsilon}.
$$

When $\mathfrak{q}=p^k$ is a prime power, Blomer and Mili\'{c}evi\'{c} \cite{BD}
proved the subconvexity bounds for twists of $\rm GL(2)$ $L$-functions:
$$
L\left(\frac{1}{2}+it,g\times \chi\right)\ll_{p,t,g,\varepsilon} \mathfrak{q}^{1/3+\varepsilon}.
$$

For the $\rm GL(3)$ case, Sun and Zhao \cite{SZ} obtained
$$
L\left(\frac{1}{2},F\times \chi\right)\ll_{F,\varepsilon}p^{3/4}
\mathfrak{q}^{3/4-3/40+\varepsilon}
$$
for any $\varepsilon>0$.

Our main result is the following.
\begin{theorem}\label{main-theorem}
Let $F$ be a Hecke-Maass cusp form for $\rm SL(3,\mathbb{Z})$ and $\chi$ be a
Dirchlet character of prime power conductor $\mathfrak{q}=p^k$, with $k\geq3$. Then we have
$$
L\left(\frac{1}{2}+it,F\times \chi\right)\ll_{F,\varepsilon}
p^{3/4}\big(\mathfrak{q}(1+|t|)\big)^{3/4-3/40+\varepsilon}.
$$
\end{theorem}

\begin{remark}
In rest of the paper, we will carry out the proof under the assumption $t\geq \mathfrak{q}^{\varepsilon}$ for
some $\varepsilon>0$. For $t\leq \mathfrak{q}^{\varepsilon}$, one can extend the method of \cite{SZ}
to prove $
L\left(\frac{1}{2}+it,F\times \chi\right)\ll_{t,\pi,\varepsilon}p^{3/4}\mathfrak{q}^{3/4-3/40+\varepsilon}
$ with polynomial dependence on $t$. For $t\leq -\mathfrak{q}^{\varepsilon}$, the same result follows from the
case $t\geq \mathfrak{q}^{\varepsilon}$ by the functional equation.
\end{remark}

\begin{remark}
We are not trying to get the best exponent in $p$. With the present exponent $3/4$, the bound in
Theorem \ref{main-theorem} breaks the convexity for $k>10$ in $\mathfrak{q}$-aspect.
\end{remark}

\medskip

\noindent{\bf Notation.}
Throughout the paper, $\varepsilon$ and $A$ are arbitrarily small and arbitrarily large positive
numbers, respectively,  which may be different at each occurrence. As usual,
$e(x)=e^{2\pi ix}$ and  the symbol $n\sim X$ means $X<n\leq 2X$.
\section{Preliminaries}
\label{prelim}
\subsection{Hecke--Maass cusp forms for $\mathrm{GL}(3)$}

Let $F$ be a Hecke--Maass cusp form for $\rm SL(3,\mathbb{Z})$, which is an
eigenfunction for all the Hecke
operators. Let the Fourier coefficients be $A(n_1,n_2)$, normalized so that $A(1, 1)=1$.
The Langlands parameters $(\mu_1, \mu_2, \mu_3)$ associated with $F$ are
$\mu_1=-\nu_1-2\nu_2+1$, $\mu_2=-\nu_1+\nu_2$, $\mu_3=2\nu_1+\nu_2-1$.

By Rankin--Selberg theory, we have
\begin{align}\label{Rankin--Selberg}
\mathop{\sum\sum}_{n_1^2n_2\leq N} \left|A(n_1,n_2)\right|^2 \ll N.
\end{align}
As \cite{HX}, we record the individual bound
\begin{align}
A(n_1,n_2)\ll(n_1n_2)^{\theta_3+\varepsilon},
\end{align}
where $\theta_3 \leq 5/14$ is the bound toward to the Ramanujan conjecture on GL(3). So we have
\begin{align}
\sum_{n_2\sim N}|A(n_1,n_2)|\ll\sum_{r|n_1^\infty}\sum_{n_2\sim N/n_1\atop(n_1,n_2)=1}|A(n_1,rn_2)|\ll
\sum_{r|n_1^\infty}|A(n_1,r)|
\sum_{n_2\sim N\atop (n_1,n_2)=1}|A(1,n_2)|\ll n_1^{\theta_3+\varepsilon}N,
\end{align}
and
\begin{align}\label{theta3}
\sum_{n_2\sim N}|A(n_1,n_2)|^2\ll\sum_{r|n_1^\infty}\sum_{n_2\sim N/n_1\atop(n_1,n_2)=1}|A(n_1,rn_2)|^2\ll
\sum_{r|n_1^\infty}|A(n_1,r)|^2
\sum_{n_2\sim N\atop (n_1,n_2)=1}|A(1,n_2)|^2\ll n_1^{2\theta_3+\varepsilon}N.
\end{align}
Here we have used \eqref{Rankin--Selberg} and the fact
$\sum_{d|n_1^\infty}d^{-\sigma}\ll n_1^{\varepsilon}$,
for $\sigma > 0$.

The $L$-function $L(s,F\times\chi)$
satisfies the functional equation
$$
\Lambda(s,F\times\chi)=\epsilon(F \times \chi)
\Lambda(1-s,\tilde{F}\times\overline{\chi}),
$$
where
$$
\Lambda(s,F\times\chi)=\mathfrak{q}^{-3s/2}\pi^{-3s/2}
\prod_{j=1}^{3}\Gamma\left(\frac{s-\mu_j}{2}\right)L(s,F \times \chi)
$$
is the completed $L$-function and $\epsilon(F \times \chi)$ is the root number.
Here $\tilde{F}$ is the dual cusp form which has Langlands parameters $(-\mu_3, -\mu_2, -\mu_1)$.
By \cite[Chapter 5.2]{IK}, we can obtain the approximate functional equation which leads
us to the following result.

\begin{lemma}\label{functional-equation}
We have
$$
L\left(\frac{1}{2}+it, F \times \chi\right)\ll\big(\mathfrak{q}(|t|+1)\big)^{\varepsilon}
\sup_{N\ll\big(\mathfrak{q}(|t|+1)\big)^{3/2+\varepsilon}}
\frac{S(N)}{\sqrt{N}}+\big(\mathfrak{q}(|t|+1)\big)^{-A},
$$
where
$$
S(N)=\sum_{n\geq 1}A(1,n)\chi(n)n^{-it}V\left(\frac{n}{N}\right),
$$
with compactly supported smooth function $V$ such that $\rm supp$ $V \subset [1,2]$ and $V^{(j)}\ll 1$ for $j\geq 1$.
\end{lemma}

\subsection{Summation formulas}

We first recall the Poisson summation formulae over
an arithmetic progression.
\begin{lemma}
Let $\beta \in \mathbb{Z}$ and $c \in \mathbb{Z}_{\geq 1}$. For a Schwartz function $f: \mathbb{R} \to \mathbb{C}$, we have
\begin{align*}
\mathop{\sum}_{n\in \mathbb{Z} \atop n \equiv \beta \bmod c}f(n)=\frac{1}{c}
\sum_{n\in\mathbb{Z}}\hat{f}\left( \frac{n}{c}\right)e\left(\frac{n\beta}{c}\right),
\end{align*}
where $\hat{f}=\int_{\mathbb{R}}f(x)e(-xy)\mathrm{d}x$ is the Fourier transform of $f$.
\begin{proof} See e.g.\cite[Eq. (4.24)]{IK}.\end{proof}
\end{lemma}

We now recall the Voronoi summation formula for $\rm SL(3, \mathbb{Z})$. For $\ell=0,1$ we define
$$
\gamma_\ell (s)=\frac{1}{2\pi^{3(s+1/2)}}\prod_{j=1}^3
\frac{\Gamma\left((1+s+\mu_j+\ell)/2\right)}
{\Gamma\left((-s-\mu_j+\ell)/2\right)}
$$
and set $\gamma_{\pm}(s)=\gamma_0(s)\mp i \gamma_1(s)$.
Here $\mu_j$ are the Langlands parameters of $F$ as above.
For $\psi(x)\in \mathcal{C}_c^\infty(0,\infty)$ we
denote by $\widetilde{\psi}(s)$
the Mellin transform of $\psi(x)$.
Let
\begin{align}\label{intgeral transform-3}
\Psi^{\pm}\left(x\right)
=\frac{1}{2\pi i}\int_{(\sigma)}x^{-s}
\gamma_{\pm}(s)\widetilde{\psi}(-s)\mathrm{d}s,
\end{align}
where $\sigma>\max\limits_{1\leq j\leq 3}\{-1-\mathrm{Re}(\mu_j)\}$.
Then we have the following Voronoi summation formula.

\begin{lemma}\label{voronoiGL3}
Let $q\in \mathbb{N}$ and $a\in \mathbb{Z}$ be such
that $(a,q)=1$. Then
$$
\sum_{n=1}^{\infty}A\left(r,n\right)
e\left(\frac{an}{q}\right)
\psi\left(n\right)= q\sum_{\pm}\sum_{n_{1}|qr}
\sum_{n_{2}=1}^{\infty}\frac{A\left(n_{2},n_{1}\right)}{n_{1}n_{2}}
S\left(r\overline{a},\pm n_{2};\frac{rq}{n_{1}}\right)
\Psi^{\pm}\left(\frac{n_{1}^{2}n_{2}}{q^{3}r}\right),
$$
where $a \overline{a} \equiv 1(\bmod\,q)$ and $S(m,n;c)$ is the classical Kloosterman sum.
\end{lemma}
\subsection{The delta method}

There are two oscillatory factors contributing to the convolution sums.
Our method is based on separating these oscillations using the $\delta$-method. In the present situation we will use a version of the
circle method by Duke, Friedlander and Iwaniec (see \cite[Chapter 20]{IK}).

Define $\delta: \mathbb{Z}\rightarrow \{0,1\}$ with
$\delta(0)=1$ and $\delta(n)=0$ for $n\neq 0$.
For any $n\in \mathbb{Z}$ and $Q\in \mathbb{R}^+$, we have
\begin{align}\label{DFI's}
\delta(n)=\frac{1}{Q}\sum_{1\leq q\leq Q} \;\frac{1}{q}\;\sideset{}{^\star}\sum_{a\bmod{q}}
e\left(\frac{na}{q}\right)\int_\mathbb{R}g(q,x) e\left(\frac{nx}{qQ}\right)\mathrm{d}x,
\end{align}
where the $\star$ on the sum indicates
that the sum over $a$ is restricted to $(a,q)=1$.
The function $g$ has the following properties
(see (20.158) and (20.159) of \cite{IK}
\footnote{After correcting a typo in eq. (20.158) there.} and \cite[Lemma 15]{Huang})
\begin{align}\label{g-h}
g(q,x)\ll |x|^{-A},\quad g(q,x) =1+
O\left(\frac{Q}{q}\left(\frac{q}{Q}+|x|\right)^A\right)
\end{align}
for any $A>1$ and
\begin{align}\label{g rapid decay}
\frac{\partial^j}{\partial x^j}g(q,x)\ll
|x|^{-j}\min\left(|x|^{-1},\frac{Q}{q}\right)\log Q, \quad j\geq 1.
\end{align}
In particular the first property in \eqref{g-h} implies that
the effective range of the integration in
\eqref{DFI's} is $[-Q^\varepsilon, Q^\varepsilon]$.
\subsection{Oscillatory integrals}

Let
\begin{align*}
 I = \int_{\mathbb{R}} w(y) e^{i \varrho(y)} dy.
\end{align*}
Firstly, we have the following estimates for exponential integrals
(see \cite[Lemma 8.1]{BKY}  and \cite[Lemma A.1]{AHLQ}).
	
	\begin{lemma}\label{lem: upper bound}
		Let $w(x)$ be a smooth function    supported on $[ a, b]$ and
        $\varrho(x)$ be a real smooth function on  $[a, b]$. Suppose that there
		are   parameters $Q, U,   Y, Z,  R > 0$ such that
		\begin{align*}
		\varrho^{(i)} (x) \ll_i Y / Q^{i}, \qquad w^{(j)} (x) \ll_{j } Z / U^{j},
		\end{align*}
		for  $i \geqslant 2$ and $j \geqslant 0$, and
		\begin{align*}
		| \varrho' (x) | \geqslant R.
		\end{align*}
		Then for any $A \geqslant 0$ we have
		\begin{align*}
		I \ll_{ A} (b - a)
Z \bigg( \frac {Y} {R^2Q^2} + \frac 1 {RQ} + \frac 1 {RU} \bigg)^A .
		\end{align*}
			\end{lemma}

Next, we need the following evaluation for exponential integrals
which are
 Lemma 8.1 and Proposition 8.2 of \cite{BKY} in the language of inert functions.

Let $\mathcal{F}$ be an index set, $Y: \mathcal{F}\rightarrow\mathbb{R}_{\geq 1}$
and under this map $T\mapsto Y_T$
be a function of $T \in \mathcal{F}$.
A family $\{w_T\}_{T\in \mathcal{F}}$ of smooth
functions supported on a product of dyadic intervals in $\mathbb{R}_{>0}^d$
is called $Y$-inert if for each $j=(j_1,\ldots,j_d) \in \mathbb{Z}_{\geq 0}^d$
we have
$$
C(j_1,\ldots,j_d)
= \sup_{T \in \mathcal{F} } \sup_{(y_1, \ldots, y_d) \in \mathbb{R}_{>0}^d}
Y_T^{-j_1- \cdots -j_d}\left| y_1^{j_1} \cdots y_d^{j_d}
w_T^{(j_1,\ldots,j_d)}(y_1,\ldots,y_d) \right| < \infty.
$$

\begin{lemma}
\label{lemma:exponentialintegral}
 Suppose that $w = w_T(y)$ is a family of $Y$-inert functions,
 with compact support on $[Z, 2Z]$, so that
$w^{(j)}(y) \ll (Z/Y)^{-j}$.  Also suppose that $\varrho$ is
smooth and satisfies $\varrho^{(j)}(y) \ll H/Z^j$ for some
$H/Y^2 \geq R \geq 1$ and all $y$ in the support of $w$.
\begin{enumerate}
 \item
 If $|\varrho'(y)| \gg H/Z$ for all $y$ in the support of $w$, then
 $I \ll_A Z R^{-A}$ for $A$ arbitrarily large.
 \item If $\varrho''(y) \gg H/Z^2$ for all $y$ in the support of $w$,
 and there exists $y_0 \in \mathbb{R}$ such that $\varrho'(y_0) = 0$
 (note $y_0$ is necessarily unique), then
 \begin{align*}
  I = \frac{e^{i \varrho(y_0)}}{\sqrt{\varrho''(y_0)}}
 F(y_0) + O_{A}(  Z R^{-A}),
 \end{align*}
where $F(y_0)$ is an $Y$-inert function (depending on $A$) supported
on $y_0 \asymp Z$.
\end{enumerate}
\end{lemma}
\section{The set-up}

We will prove the following propositon, from which we prove
Theorem \ref{main-theorem} by using Lemma \ref{functional-equation}.
\begin{proposition}\label{upper bound}
We have
$$
S(N)\ll p^{3/4}N^{1/2+\varepsilon}(\mathfrak{q}t)^{3/4-3/40}.
$$
for $N\ll(\mathfrak{q}t)^{3/2+\varepsilon}$.
\end{proposition}

Recall that
$$
S(N)=\sum_{n\geq 1}A(1,n)\chi(n)n^{-it}V\left(\frac{n}{N}\right),
$$
where $\chi$ is a primitive Dirichlet character modulo $\mathfrak{q}=p^k$.
In order to use the delta symbol method, we rewrite the $\delta(n-m)$
in a more analytic form in the following lemma.

\begin{lemma}\label{circle method}
Let $1 \leq \lambda \leq k$. Then we have
\begin{align*}
\delta(n)
=&\sum_{r=0}^{\lambda}\frac{1}{Q}\sum_{q\leq Q\atop (q,p)=1}\frac{1}{qp^\lambda}
\;\sideset{}{^\star}\sum_{a\bmod qp^{\lambda-r}}
e\left(\frac{an}{qp^{\lambda-r}}\right)
\int_{\mathbb{R}}
  g(q,x)e\left(\frac{nx}{Qqp^\lambda}\right)\mathrm{d}x\\
&+\sum_{s=1}^{[\log Q/\log p]}
\frac{1}{Q}\sum_{q\leq Q/p^s\atop (q,p)=1}\frac{1}{qp^{\lambda+s}}
\sideset{}{^\star}\sum_{a\bmod qp^{\lambda+s}}
e\left(\frac{an}{qp^{\lambda+s}}\right)
\int_{\mathbb{R}}
  g(p^sq,x)e\left(\frac{nx}{Qqp^{\lambda+s}}\right)\mathrm{d}x.
\end{align*}
\end{lemma}

\begin{proof}

Define $\mathbf{1}_\mathscr{F}=1$ if $\mathscr{F}$ is true, and is 0 otherwise.
By \eqref{DFI's}, we write $\delta(n)$ as
$\delta(n/p^\lambda)\mathbf{1}_{p^\lambda|n}$ and detect the congruence by additive
characters to get
\begin{align*}
\delta(n)=\frac{1}{Q}\sum_{q\leq Q}\frac{1}{qp^\lambda}
\sum_{b\bmod p^\lambda}\;
\sideset{}{^\star}\sum_{a\bmod q}e\left(\frac{a+bq}{qp^\lambda}n\right)
\int_{\mathbb{R}}g(q,x)e\left(\frac{nx}{Qqp^\lambda}\right)\mathrm{d}x.
\end{align*}
which can be further written as $\delta_1(n)+\delta_2(n)$ with
\begin{align*}
\delta_1(n)&=\frac{1}{Q}\sum_{q\leq Q\atop (q,p)=1}\frac{1}{qp^\lambda}
\sum_{b\bmod p^\lambda}\;
\sideset{}{^\star}\sum_{a\bmod q}e\left(\frac{a+bq}{qp^\lambda}n\right)
\int_{\mathbb{R}}
  g(q,x)e\left(\frac{nx}{Qqp^\lambda}\right)\mathrm{d}x,\\
\delta_2(n)&=\frac{1}{Q}\sum_{q\leq Q \atop p|q}\frac{1}{qp^{\lambda}}
\sum_{b\bmod p^\lambda}\;
\sideset{}{^\star}\sum_{a\bmod q}e\left(\frac{a+bq}{qp^{\lambda}}n\right)
\int_{\mathbb{R}}
  g(q,x)e\left(\frac{nx}{Qqp^{\lambda}}\right)\mathrm{d}x.
\end{align*}
For $\delta_1(n)$, making a change of variable $a\to ap^{\lambda}$, we have
\begin{align*}
\delta_1(n)
=&\frac{1}{Q}\sum_{q\leq Q\atop (q,p)=1}\frac{1}{qp^\lambda}\;
\sideset{}{^\star}\sum_{b\bmod p^{\lambda}}\;
\sideset{}{^\star}\sum_{a\bmod q}
e\left(\frac{ap^{\lambda}+bq}{qp^{\lambda}}n\right)
\int_{\mathbb{R}}
  g(q,x)e\left(\frac{nx}{Qqp^\lambda}\right)\mathrm{d}x\\
&+\frac{1}{Q}\sum_{q\leq Q\atop (q,p)=1}\frac{1}{qp^\lambda}
\sum_{b\bmod p^{\lambda-1}}\;
\sideset{}{^\star}\sum_{a\bmod q}
e\left(\frac{ap^{\lambda-1}+bq}{qp^{\lambda-1}}n\right)
\int_{\mathbb{R}}
  g(q,x)e\left(\frac{nx}{Qqp^\lambda}\right)\mathrm{d}x.
\end{align*}

Observe that in the first sum, $a$ varies over a set of representatives of
the residue classes modulo $q$ (prime to $q$) and $b$ varies over a set of representatives of the residue
classes modulo $p^{\lambda}$, $ap^{\lambda} + bq$ varies over a set of representatives of the
residue classes modulo $qp^{\lambda}$ prime to $qp^{\lambda}$. Then repeating the process one can get
\begin{align*}
\delta_1(n)
=&\frac{1}{Q}\sum_{q\leq Q\atop (q,p)=1}\frac{1}{qp^\lambda}\;
\sideset{}{^\star}\sum_{a\bmod qp^{\lambda}}
e\left(\frac{an}{qp^\lambda}\right)
\int_{\mathbb{R}}
  g(q,x)e\left(\frac{nx}{Qqp^\lambda}\right)\mathrm{d}x\\
&+\frac{1}{Q}\sum_{q\leq Q\atop (q,p)=1}\frac{1}{qp^\lambda}
\sum_{b\bmod p^{\lambda-1}}\;
\sideset{}{^\star}\sum_{a\bmod q}
e\left(\frac{ap^{\lambda-1}+bq}{qp^{\lambda-1}}n\right)
\int_{\mathbb{R}}
  g(q,x)e\left(\frac{nx}{Qqp^\lambda}\right)\mathrm{d}x\\
=&\sum_{r=0}^{\lambda}\frac{1}{Q}\sum_{q\leq Q\atop (q,p)=1}\frac{1}{qp^\lambda}
\;\sideset{}{^\star}\sum_{a\bmod qp^{\lambda-r}}
e\left(\frac{na}{qp^{\lambda-r}}\right)
\int_{\mathbb{R}}
  g(q,x)e\left(\frac{nx}{Qqp^\lambda}\right)\mathrm{d}x.
\end{align*}
For $\delta_2(n)$, similarly making a change of variable $q\to qp$, we have
\begin{align*}
\delta_2(n)
=&\frac{1}{Q}\sum_{q\leq Q/p\atop (q,p)=1}\frac{1}{qp^{\lambda+1}}
\sum_{b\bmod p^\lambda}\;
\sideset{}{^\star}\sum_{a\bmod qp}e\left(\frac{a+bpq}{qp^{\lambda+1}}n\right)
\int_{\mathbb{R}}
  g(pq,x)e\left(\frac{nx}{Qqp^{\lambda+1}}\right)\mathrm{d}x\\
&+\frac{1}{Q}\sum_{q\leq Q/p^2}\frac{1}{qp^{\lambda+2}}
\sum_{b\bmod p^\lambda}\;
\sideset{}{^\star}\sum_{a\bmod qp^2}e\left(\frac{a+bp^2q}{qp^{\lambda+2}}n\right)
\int_{\mathbb{R}}
  g(p^2q,x)e\left(\frac{nx}{Qqp^{\lambda+2}}\right)\mathrm{d}x.
\end{align*}

Similarly  $a$ varies over a set of
representatives of the residue classes modulo $qp$ (prime
to $qp$) and $b$ varies over a set of representatives
of the residue classes modulo $p^{\lambda}$, $a+bpq$
varies over a set of representatives of the residue
classes modulo $qp^{\lambda+1}$ prime to $qp^{\lambda+1}$. Then
\begin{align*}
\delta_2(n)
=&\frac{1}{Q}\sum_{q\leq Q/p\atop (q,p)=1}\frac{1}{qp^{\lambda+1}}
\sideset{}{^\star}\sum_{a\bmod qp^{\lambda+1}}
e\left(\frac{an}{qp^{\lambda+1}}\right)
\int_{\mathbb{R}}
  g(pq,x)e\left(\frac{nx}{Qqp^{\lambda+1}}\right)\mathrm{d}x\\
&+\frac{1}{Q}\sum_{q\leq Q/p^2}\frac{1}{qp^{\lambda+2}}
\sum_{b\bmod p^\lambda}\;
\sideset{}{^\star}\sum_{a\bmod qp^2}e\left(n\frac{a+bp^2q}{qp^{\lambda+2}}\right)
\int_{\mathbb{R}}
  g(p^2q,x)e\left(\frac{nx}{Qqp^{\lambda+2}}\right)\mathrm{d}x\\
=&\sum_{s=1}^{[\log Q/\log p]}
\frac{1}{Q}\sum_{q\leq Q/p^s\atop (q,p)=1}\frac{1}{qp^{\lambda+s}}
\sideset{}{^\star}\sum_{a\bmod qp^{\lambda+s}}
e\left(\frac{an}{qp^{\lambda+s}}\right)
\int_{\mathbb{R}}
  g(p^sq,x)e\left(\frac{nx}{Qqp^{\lambda+s}}\right)\mathrm{d}x.
\end{align*}
This proves the lemma.
\end{proof}

\medskip

Now we write
$$
S(N)=\sum_{n\geq 1}A(1, n)W\left(\frac{n}{N}\right)\\
\sum_{m \geq 1\atop p^{\lambda}|m-n}\chi(m)m^{-it}
V\left(\frac{m}{N}\right)\delta\left(\frac{n-m}{p^\lambda}\right),
$$
with compactly supported smooth function
$W$ such that $\rm supp$ $W \subset [1,2]$ and $W^{(j)}\ll 1$ for $j\geq 1$.
Applying Lemma \ref{circle method} with $\lambda\in\mathbb{N}\ (2\leq \lambda \leq k)$
being a parameter to be determined later, we have
$$S(N)\ll \mathfrak{q}^{\varepsilon} \left|D(N)\right|,$$
where
\begin{align*}
D(N)=&\sum_{n\geq1}A(1, n)W\left(\frac{n}{N}\right)\sum\limits_{m\geq1}\chi(m)m^{-it}
V\left(\frac{m}{N}\right)e\left(\frac{(n-m)b}{p^\lambda}\right)\\
&\cdot\frac{1}{Q}\sum_{1\leq q\leq Q} \;\frac{1}{qp^\lambda}\;\sideset{}{^\star}\sum_{a\bmod{qp^\lambda}}
e\left(\frac{(n-m)a}{qp^\lambda}\right)\int_\mathbb{R}g(q,x) e\left(\frac{(n-m)x}{Qqp^\lambda}\right)\mathrm{d}x.
\end{align*}
Exchanging the order of integration and summations we get
\begin{align*}
D(N)=&\frac{1}{Q}
\sum_{1\leq q\leq Q \atop (q,p)=1}\frac{1}{qp^{\lambda}}
\;\sideset{}{^\star}\sum_{a\bmod{qp^\lambda}}
\int_\mathbb{R}g(q, x)\sum\limits_{m\geq1}\chi(m)
e\left(-\frac{am}{qp^{\lambda}}\right)m^{-it}
V\left(\frac{m}{N}\right)e\left(-\frac{mx}{Qqp^\lambda}\right)\\
&\qquad\qquad\qquad\qquad\cdot\sum_{n\geq1}A(1, n)
e\left(\frac{an}{qp^{\lambda}}\right)
W\left(\frac{n}{N}\right)
e\left(\frac{nx}{Qqp^\lambda}\right)\mathrm{d}x.
\end{align*}
Inserting a smooth partition of unity for the $x$-integral  and a dyadic partition for the $q$-sum, we get
$$
D(N)\ll N^\varepsilon \sup_{t^{-B}\ll X\ll t^\varepsilon}\sup_{1\ll R\ll Q}|D(N, X, R)|+O(t^{-A}),
$$
for any large positive constant $A$ and some large constant $B>0$ depending on $A$,
where
$$
\begin{aligned}
D(N, X, R)=&\frac{1}{Q}
\sum_{q\sim R\atop (q,p)=1}\frac{1}{qp^\lambda}
\;\sideset{}{^\star}\sum_{a\bmod{qp^\lambda}}
\int_\mathbb{R}g(q, x)U\left(\frac{\pm x}{X} \right)\\
&\cdot\sum\limits_{n\geq1}A(1, n)
e\left(\frac{an}{qp^{\lambda}}\right)
W\left(\frac{n}{N}\right)e\left(\frac{nx}{Qqp^\lambda}\right)\\
&\cdot\sum\limits_{m\geq1}\chi(m)
e\left(-\frac{am}{qp^{\lambda}}\right)m^{-it}
V\left(\frac{m}{N}\right)e\left(-\frac{mx}{Qqp^\lambda}\right)
\mathrm{d}x.
\end{aligned}
$$
We denote $m$-sum and $n$-sum by $\mathfrak{A}$ and $\mathfrak{B}$, respectively.
\section{Applying Poisson and Voronoi}

 In this section we transform $\mathfrak{A}$ and $\mathfrak{B}$ by the Poisson summation formula and the $\rm GL(3)$
Voronoi formula, respectively, and obtain the following results.
\begin{lemma}\label{A}
We have
$$
\mathfrak{A}=
\frac{N^{1-it}\tau(\chi)\chi(q)}{p^k}
\sum\limits_{m\equiv ap^{k-\lambda}\bmod q}\bar{\chi}
\left(m-ap^{k-\lambda}\right)\mathfrak{J}(m,q,x),
$$
where $\tau(\chi)$ is the Gauss sum and
$$
\mathfrak{J}(m,q,x)=\int_{\mathbb{R}}V(y)y^{-it}
e\left(-\frac{Nxy}{Qqp^\lambda}-\frac{mNy}{qp^k} \right)\mathrm{d}y.
$$
\end{lemma}
\begin{proof}
Applying Poisson summation with modulus $qp^k$ on the $m$-sum, we get
\begin{align*}
\mathfrak{A}&=\sum\limits_{\beta \bmod qp^k}\chi(\beta)e\left(-\frac{a\beta}{qp^\lambda}\right)
\sum\limits_{m\equiv\beta \bmod qp^k}m^{-it}V\left(\frac{m}{N}\right)e\left(-\frac{mx}{Qqp^\lambda}\right)\\
&=\frac{N^{1-it}}{qp^k}\sum\limits_{m\in \mathbb{Z}}\sum\limits_{\beta \bmod qp^k}\chi(\beta)e\left(\frac{m-ap^{k-\lambda}}{qp^k}\beta\right)
\int_{\mathbb{R}}V(y)y^{-it}
e\left(-\frac{Nxy}{Qqp^\lambda}-\frac{mNy}{qp^k} \right)\mathrm{d}y.
\end{align*}
Since $(q,p)=1$, the sum $\sum\limits_{\beta \bmod qp^k}\chi(\beta)e\left(\frac{m-ap^{k-\lambda}}{qp^k}\beta\right)$
factors as
\begin{align*}
&\sum\limits_{\beta \bmod q} e\left(\frac{m-ap^{k-\lambda}}{q}\beta\right)\times\sum\limits_{\beta \bmod p^k}\chi(q\beta)e\left(\frac{m-ap^{k-\lambda}}{p^k}\beta\right)\\
&=q \,\delta(m\equiv ap^{k-\lambda} \bmod q)\,\chi(q)\tau(\chi)\,\overline{\chi}(m-ap^{k-\lambda}).
\end{align*}
Now completes the proof.
\end{proof}

\begin{lemma}\label{B}
We have
$$
\mathfrak{B}= qp^\lambda\sum\limits_{\eta=\pm1}\sum\limits_{n_1|qp^\lambda}\sum\limits_{n_2}
\frac{A(n_2,n_1)}{n_1n_2}S\left(\overline{a},\eta n_2;\frac{qp^\lambda}{n_1}\right)\Psi_{x}^{\mathrm{sgn}(\eta)}
\left(\frac{n_1^2 n_2}{q^3p^{3\lambda}}\right),
$$
where $\Psi_{x}^{\mathrm{sgn}(\eta)}(z)$ is defined as in Lemma \ref{voronoiGL3} with $\psi(y)$ replaced by
$W(\frac{y}{N})e\left(\frac{xy}{Qqp^\lambda}\right)$.
\end{lemma}
By Lemma \ref{A} and Lemma \ref{B},
the main sum of $D(N,X,R)$ can be expressed as
\begin{align}\label{s0}
&\frac{N^{1-it}\tau(\chi)}{Qp^{k}}\sum_{q\sim R \atop (q, p)=1}\chi(q)
\;\sideset{}{^\star}\sum_{a\bmod{qp^\lambda}}
\sum_{\eta=\pm1}\sum_{n_1|qp^\lambda}\sum_{n_2}
\frac{A(n_2,n_1)}{n_1n_2}\notag \\
&\cdot\sum_{m\equiv ap^{k-\lambda}\bmod q}\overline{\chi}
\left(m-ap^{k-\lambda}\right)
S\left(\overline{a},\eta n_2;\frac{qp^\lambda}{n_1}\right)\notag\\
&\cdot\int_{\mathbb{R}}g(q,x)U\left(\frac{\pm x}{X}\right)
\Psi_{x}^{\mathrm{sgn}(\eta)}
\left(\frac{n_1^2 n_2}{q^3p^{3\lambda}}\right)
\mathfrak{J}(m,q,x)
\mathrm{d}x.
\end{align}

\begin{lemma}\label{INT}
We have
\begin{enumerate}
  \item If $zN\gg t^\varepsilon$, then $\Psi_{x}^{\eta}(z)$ is negligibly small unless $\mathrm{sgn}(x)=-\mathrm{sgn}(\eta)$ and $\frac{-\eta Nx}{p^\lambda qQ}\asymp(zN)^{1/3}$, in which case we have
      $$
      \Psi_{x}^{\mathrm{sgn}(\eta)}(z)=(zN)^{1/2}e\left( \eta\frac{2{(zp^\lambda qQ)}^{1/2}}{{(-\eta x)}^{1/2}}\right)
      \mathcal{W}\left(\frac{z^{1/2}(p^\lambda qQ)^{3/2}}{N(-\eta x)^{3/2}} \right)+O(t^{-A}),
      $$
      where $\mathcal{W}$ is a certain compactly supported 1-inert function depending on $A$.
  \item If $zN\ll t^\varepsilon$, and $\frac{NX}{p^\lambda RQ}\gg t^\varepsilon$, then $\Psi_{x}^{\mathrm{sgn}(\eta)}(z)\ll t^{-A}$.
  \item If $zN\ll t^\varepsilon$, and $\frac{NX}{p^\lambda RQ}\ll t^\varepsilon$, then $\Psi_{x}^{\mathrm{sgn}(\eta)}(z)\ll t^\varepsilon$.
\end{enumerate}
\end{lemma}
\begin{proof} See ~\cite[5.3]{Huang}.\end{proof}

We first consider the oscillating cases and we will treat
the the non-oscillating case in Section \ref{non-oscillating}.
From now on we make assumption: $\frac{NX}{p^{\lambda}RQ}\gg t^{\varepsilon}$. In this case,
by Lemma \ref{INT}, \eqref{s0} is reduced to
\begin{align}\label{Key}
&\ \ \ \frac{N^{3/2-it}\tau(\chi)}{Qp^{k+3\lambda/2}}\sum_{q\sim R}\frac{\chi(q)}{q^{3/2}}\;\sideset{}{^\star}\sum_{a\bmod{qp^\lambda}}
\sum_{\eta=\pm1}\sum_{n_1|qp^\lambda}\sum_{n_2\asymp\frac{N_1}{n_1^2}}
\frac{A(n_2,n_1)}{\sqrt{n_2}} \notag \\
&\cdot\sum_{m\equiv ap^{k-\lambda}\bmod q}\overline{\chi}
\left(m-ap^{k-\lambda}\right)
S\left(\overline{a},\eta n_2;\frac{qp^\lambda}{n_1}\right)
\int_{\mathbb{R}}V(y)e\left(-\frac{t\log y}{2\pi}-\frac{mNy}{qp^k}\right)\notag\\
&\cdot\int_{\mathbb{R}}g(q,x)e\left(-\frac{Nxy}{Qqp^{\lambda}}
+\eta\frac{2(n_1^2n_2Q)^{1/2}}{(-\eta x)^{1/2}qp^\lambda}\right)
\mathcal{W}\left(\frac{(n_1^2n_{2})^{1/2}Q^{3/2}}{N(-\eta x)^{3/2}}\right)
U\left(\frac{-\eta x}{X}\right)\mathrm{d}x\mathrm{d}y,
\end{align}
where $N_1=\frac{N^2X^3}{Q^3}$. Let $x=-\eta Xv$,
then the resulting $x$-integral in \eqref{Key} becomes
\begin{align}\label{x-integral}
-\eta X\int_{\mathbb{R}}g(q,-\eta Xv)
e\left(\eta\frac{NXyv}{Qqp^{\lambda}}+
\eta\frac{2(n_1^2n_2Q)^{1/2}}{(Xv)^{1/2}qp^\lambda}\right)
\mathcal{W}\left(\frac{(n_1^2n_{2})^{1/2}Q^{3/2}}{N(Xv)^{3/2}}\right)U(v)\mathrm{d}v.
\end{align}
Let
$$
h(v)=\eta\frac{NXyv}{Qqp^{\lambda}}+\eta\frac{2(n_1^2n_2Q)^{1/2}}{(Xv)^{1/2}qp^\lambda}.
$$
Then
$$
h'(v)=\eta\frac{NXy}{Qqp^{\lambda}}-\eta\frac{(n_1^2n_2Q)^{1/2}}{X^{1/2}qp^\lambda}
v^{-3/2}, \quad h''(v)=\eta\frac{3(n_1^2n_2Q)^{1/2}}{2X^{1/2}qp^\lambda}v^{-5/2}.
$$
Note that the solution of $h'(v_0)=0$ is $v_0=\frac{(n_1^2n_2)^{1/3}
Q}{X(Ny)^{2/3}}\asymp 1$, and
$$
h(v_0)=\eta
\frac{3(n_1^2n_2Ny)^{1/3}}{qp^{\lambda}},\quad h''(v_0)=
\frac{3\eta}{2v_0^2}\frac{(n_1^2n_2Q)^{1/2}}
{qp^{\lambda}(Xv_0)^{1/2}}=\frac{3\eta}{2v_0^2}
\frac{(n_1^2n_2Ny)^{1/3}}{qp^{\lambda}}.
$$
By \eqref{g rapid decay}, we have \eqref{x-integral} is equal to
\begin{align}
\frac{(qp^{\lambda})^{1/2}X}{(n_1^2n_2Ny)^{1/6}}
e\left(\eta\frac{3(n_1^2n_2Ny)^{1/3}}{qp^{\lambda}}\right)
g(q,-\eta Xv_0)
\mathcal{U}(v_0)
\mathcal{W}\left(\frac{Q^{3/2}(n_1^2n_{2})^{1/2}}
{N(Xv_0)^{3/2}}\right)+O(t^{-A}),
\end{align}
where $\mathcal{U}$ is a certain compactly supported 1-inert function depending on A. Hence, by letting
$\mathcal{V}(y)=y^{-1/6}V(y)g(q,-\eta Xv_0)\mathcal{U}(v_0)\mathcal{W}\left(\frac{Q^{3/2}(n_1^2n_{2})^{1/2}}
{N(Xv_0)^{3/2}}\right)$,
at the cost of a negligible error, we
can rewrite \eqref{Key} as
\begin{align}
&\frac{N^{4/3-it}X\tau(\chi)}{Qp^{k+\lambda}}
\sum_{q\sim R}\frac{\chi(q)}{q}\;\sideset{}{^\star}\sum_{a\bmod{qp^\lambda}}
\sum_{\eta=\pm1}\sum_{n_1|qp^\lambda}\frac{1}{n_1^{1/3}}
\sum_{n_2\asymp\frac{N_1}{n_1^2}}
\frac{A(n_2,n_1)}{n_2^{2/3}} \notag \\
&\quad\cdot\sum_{m\equiv ap^{k-\lambda}\bmod q}\overline{\chi}
\left(m-ap^{k-\lambda}\right)
S\left(\overline{a},\eta n_2;\frac{qp^\lambda}{n_1}\right)\notag\\
&\quad\cdot\int_{\mathbb{R}}\mathcal{V}(y)e\left(-\frac{t\log y}{2\pi}+\eta\frac{3(n_1^2n_2Ny)^{1/3}}{qp^\lambda}
-\frac{mNy}{qp^k}\right)\mathrm{d}y.
\end{align}
By Lemma \ref{lem: upper bound}, the above integral is negligibly small if $mN/qp^k \gg t^{\varepsilon}\max\{t,(n_1^2n_2N)^{1/3}/qp^\lambda\}$. Recall
$\frac{NX}{p^{\lambda}RQ}\asymp \left(\frac{n_1^2n_2N}{(qp^\lambda)^{3}}\right)^{1/3}$ and $q\sim R$. Thus we only need to consider $m\ll t^{\varepsilon} \max\left\{\frac{tRp^k}{N},\frac{Xp^{k-\lambda}}{Q}\right\}$.
Then we have
\begin{align}
D(N,X,R)\ll N^\varepsilon \sup\limits_{M \ll t^{\varepsilon} \max\left\{\frac{tRp^k}{N},\frac{Xp^{k-\lambda}}{Q}\right\}}
\Big|D(N, X, R, M)\Big|+O(t^{-A}),
\end{align}
where
\begin{align}
D(N, X, R,M)=&\frac{N^{4/3}X}{Qp^{k/2+\lambda}}
\sum_{\eta=\pm1}\sum_{q\sim R\atop(q,p)=1}\frac{\chi(q)}{q}
\sum_{n_1|qp^{\lambda}}\frac{1}{n_1^{1/3}}
\sum_{n_2\asymp\frac{N_1}{n_1^2}}
\frac{A(n_2,n_1)}{n_2^{2/3}} \notag \\
&\cdot\sum_{m\sim M}\sideset{}{^\star}\sum_{a\bmod{qp^\lambda}\atop m\equiv ap^{k-\lambda}\bmod q}\overline{\chi}
\left(m-ap^{k-\lambda}\right)
S\left(\overline{a},\eta n_2;\frac{qp^{\lambda}}{n_1}\right)\notag\\
&\cdot\int_{\mathbb{R}}\mathcal{V}(y)e\left(-\frac{t\log y}{2\pi}+\eta\frac{3(n_1^2n_2Ny)^{1/3}}{qp^{\lambda}}
-\frac{mNy}{qp^k}\right)\mathrm{d}y.
\end{align}

\section{Applying Cauchy and Poisson}

Now we consider $D(N, X, R,M)$. Note that $N_1\asymp\frac{N^2X^3}{Q^3}$.
Then we use the Cauchy-Schwartz inequality
and \eqref{theta3} with $n_1=n_1'n_1''$ , $n_1'|q, \; n_1''|p^\lambda$ to get
$$
D(N, X, R,M)\ll\frac{NX^{1/2}}{Q^{1/2}p^{k/2+\lambda}}
\sum\limits_{\eta=\pm1}
\sum\limits_{(n_{1}',p)=1}\sum\limits_{n_{1}''|p^{\lambda}}
(n_{1}'n_{1}'')^{\theta_3}\;\Omega^{1/2},
$$
with
\begin{align*}
\Omega=\sum\limits_{n_2}W\left(\frac{n_{1}'^2n_{1}''^2n_2}{N_1}\right)
\Bigg|\sum_{q\sim R \atop n_1'|q, (q,p)=1}\frac{1}{q}&\sum_{m\sim M\atop (m,q)=1}
\mathcal{J}(n_{1}'^2n_{1}''^2n_2,m,q)\mathcal{C}(m,q,n_1',n_1'',n_2)
\Bigg|^2,
\end{align*}
where
$$
\mathcal{J}(n_{1}'^2n_{1}''^2n_2,m,q)=\int_{\mathbb{R}}\mathcal{V}(y)e\left(-\frac{t\log y}{2\pi}+\eta\frac{3(n_{1}'^2n_{1}''^2n_2Ny)^{1/3}}{qp^{\lambda}}
-\frac{mNy}{qp^k}\right)\mathrm{d}y,
$$
and
$$
\mathcal{C}(m,q,n_1',n_1'',n_2)=\sideset{}{^\star}\sum_{a\bmod{qp^\lambda}\atop m\equiv ap^{k-\lambda}\bmod q}\overline{\chi}\left(m-ap^{k-\lambda}\right)
S\left(\overline{a},\eta n_2;\frac{qp^{\lambda}}{n_1'n_1''}\right).
$$
Note that
$$
S\left(\overline{a},\eta n_2;\frac{qp^{\lambda}}{n_1'n_1''}\right)
=S\left(\overline{a}\overline{p^{\lambda}/n_1''},\eta n_2\overline{p^{\lambda}/n_1''};\frac{q}{n_1'}\right)
S\left(\overline{a}\overline{q/n_1'},\eta n_2\overline{q/n_1'};\frac{p^{\lambda}}{n_1''}\right)
,
$$
we have
\begin{align}\label{character}
\mathcal{C}(m,q,n_1',n_1'',n_2)
=&\sideset{}{^\star}\sum_{a\bmod{q}\atop m\equiv ap^{k-\lambda} \bmod q}
S\left(\overline{a}\overline{p^{\lambda}/n_1''},\eta n_2\overline{p^{\lambda}/n_1''};\frac{q}{n_1'}\right)\nonumber\\
\cdot&\sideset{}{^\star}\sum_{c\bmod{p^\lambda}}\overline{\chi}(m-cp^{k-\lambda})
 S\left(\overline{c}\overline{q/n_1'},\eta n_2\overline{q/n_1'};\frac{p^{\lambda}}{n_1''}\right)\nonumber\\
 :=&\sideset{}{^\star}\sum_{a\bmod{q}\atop m\equiv ap^{k-\lambda} \bmod q}
 \mathfrak{C}(m,a,q,n_1',n_1'',n_2).
\end{align}

Opening the absolute square in (5.1), we get
\begin{align}\label{Omega1}
\Omega\ll
\sum_{q\sim R \atop n_1'|q, (q,p)=1}
\sum_{q'\sim R \atop n_1'|q', (q',p)=1}\frac{1}{qq'}
\sum_{m\sim M\atop (m,q)=1}
\sum_{m'\sim M\atop (m',q')=1}
\;\sideset{}{^\star}\sum_{a\bmod{q}\atop m\equiv ap^{k-\lambda}\bmod q}
\;\sideset{}{^\star}\sum_{a'\bmod{q'}\atop m'\equiv a'p^{k-\lambda}\bmod q'}\mathcal{T}.
\end{align}
where
\begin{align*}
\mathcal{T}=\sum\limits_{n_2}W&\left(\frac{n_{1}'^2n_{1}''^2n_2}{N_1}\right)
\mathfrak{C}(m,a,q,n_1',n_1'',n_2)
\overline{\mathfrak{C}(m',a',q',n_1',n_1'',n_2)}
\\&\cdot\mathcal{J}(n_{1}'^2n_{1}''^2n_2,m,q)
\overline{\mathcal{J}(n_{1}'^2n_{1}''^2n_2,m',q')}.
\end{align*}

Let $\widehat{q}=q/n_1'$, $\widehat{q'}=q'/n_1'$
and $\widehat{p}=p^{\lambda}/n_1''$.
Applying the Poisson summation with modulus
$\widehat{q}\widehat{q'}\widehat{p}$,
we arrive at
\begin{align}\label{T1}
\mathcal{T}=\frac{N_1}{\widehat{q}\widehat{q'}\widehat{p}n_{1}'^2n_{1}''^2}
\sum\limits_{n_2}\mathfrak{C}^*(n_2)\mathcal{H}(n_2),
\end{align}
where
\begin{align}\label{Cn2}
\mathfrak{C}^*(n_2)=\sum\limits_{\beta\bmod\widehat{q}\widehat{q'}\widehat{p}}
\mathfrak{C}(m,a,q,n_1',n_1'',\beta)
\overline{\mathfrak{C}(m',a',q',n_1',n_1'',\beta)}
e\left(\frac{\eta n_2\beta}{\widehat{q}\widehat{q'}\widehat{p}}\right),
\end{align}
and
\begin{align}\label{H1(n2)}
\mathcal{H}(n_2)=\int_\mathbb{R}W(u)
\mathcal{J}(N_1u,m,q)
\overline{\mathcal{J}(N_1u,m',q')}
e\left(\frac{-n_2N_1u}{qq'p^{\lambda}n_1''} \right)\mathrm{d}u.
\end{align}

\subsection{$\mathbf{\frac{NX}{p^\lambda RQ}\ll t^{1-\varepsilon}}$}

We first consider
\begin{align*}
\mathcal{J}(N_1u, m, q)=\int_{\mathbb{R}}\mathcal{V}(y)
e\left(-\frac{t\log y}{2\pi}+\eta\frac{3(N_1uNy)^{1/3}}{qp^{\lambda}}
-\frac{mNy}{qp^k}\right)\mathrm{d}y.
\end{align*}
Write
\begin{align*}
g(y)=-\frac{t\log y}{2\pi}+Dy+3C(uy)^{1/3},
\end{align*}
where $u\in [2/3,3]$,
\begin{align}\label{DC}
D=-\frac{mN}{qp^k} \qquad
C=\eta\frac{(N_1N)^{1/3}}{qp^{\lambda}}.
\end{align}
Note that $C\asymp \frac{NX}{p^\lambda RQ}\ll t^{1-\varepsilon}$. We have
$$g'(y)=-\frac{t}{2\pi}y^{-1}+D+Cu^{1/3}y^{-2/3}\gg t+|D|,$$
then the integral $\mathcal{J}(N_1u, m, q)$ is negligibly small unless $D\asymp t$.
The stationary point $y_*$ which
is the solution to the equation $g'(y)=0$, i.e., $-\frac{t}{2\pi}y^{-1}+D+Cu^{1/3}y^{-2/3}=0$
can be written as $y_0+y_1+y_2+\cdots$ with
\begin{align*}
y_0&=\frac{t}{2\pi D}\asymp 1,\\
y_1&=-\frac{2\pi Cu^{1/3}}{t}y_0^{4/3}\asymp \frac{C}{t},\\
y_2&=\frac{4\pi^2 C^2u^{2/3}}{3t^2}y_0^{5/3}\asymp
\left(\frac{C}{t}\right)^2,\\
y_{j}&=f_{j}\left(y_0\right)\left(\frac{Cu^{1/3}}{t}\right)^{j}
\ll \left(\frac{C}{t}\right)^{j}, \quad j\geq 3.
\end{align*}
Here $y_0$ satisfies that $-\frac{t}{2\pi}y_0^{-1}+D=0$, $y_1$
satisfies that $-\frac{t}{2\pi}y_0^{-1}(1+(-1)y_1/y_0-1)+Cu^{1/3}y_0^{-2/3}=0$,
$y_2$ satisfies that $\frac{t}{2\pi}y_0^{-1}\left(
y_0^{-1}\left(\frac{y_1}{y_0}\left(1-\frac{y_1}{y_0}\right)
-\frac{y_1}{y_0}\right)+\frac{y_2}{y_0}\right)
+Cu^{1/3}y_0^{-2/3}(1+\left(-\frac{2}{3}\right)\frac{y_1}{y_0}-1)=0$,
and $f_{j}(y_0)\asymp 1$ is a function of polynomially growth, depending only on
$j$.

Recall that $\mathcal{V}(y)=y^{-1/6}V(y)g(q,-\eta Xv_0)\mathcal{U}(v_0)\mathcal{W}\left(\frac{Q^{3/2}(n_1^2n_{2})^{1/2}}
{N(Xv_0)^{3/2}}\right)$, $v_0=\frac{Q(n_1^2n_2)^{1/3}
}{X(Ny)^{2/3}}\asymp 1$, and \eqref{g rapid decay}.
So it is easy to check the conditions in Lemma \ref{lemma:exponentialintegral}.
By using the Taylor expansion,
we have
$$
g(y_*)=-\frac{t}{2\pi} \log y_0+Dy_0+g_1(D)Cu^{1/3}
+g_2(D)C^2u^{2/3}+O\left(\frac{|C|^3}{t^2}\right),
$$
for the functions $g_1(D)=3y_0^{1/3}\asymp 1$
and $g_2(D)=-\frac{4\pi}{9t}y_0^{2/3}\ll\frac{1}{t}$.
Note $g''(y_0)\asymp t$, $\mathcal{J}(N_1u, m, q)$ is essentially reduced to
\begin{align}\label{J-stationary phase-2}
\mathcal{J}(N_1u, m,q)=\frac{1}{\sqrt{t}}y_0^{-it}e\left(Cu^{\frac{1}{3}}g_1(D)
+C^2u^{\frac{2}{3}}g_2(D)
+O\left(\frac{|C|^3}{t^2}\right)\right).
\end{align}
To estimate $\mathcal{H}(n_2)$, we use the
 strategy in \cite{HX} and \cite{LS} to get
\begin{lemma}\label{integral-lemma1}
Let $N_3=\frac{Q^2Rn_1''}{NX^2}t^{\varepsilon}$ and $N_3'=t^{\varepsilon}
\left(\frac{Q^3R^2p^{\lambda}n_1''}{N^2X^3}+\frac{Nn_1''}{t^2Rp^{2\lambda}}\right)$.
Assume $\frac{NX}{p^\lambda RQ}\ll t^{1-\varepsilon}$.
\begin{enumerate}
 \item  We have $\mathcal{H}(n_2) \ll t^{-A}$ unless $n_2 \ll N_3$, in which case one has
$$\mathcal{H}(n_2) \ll \frac{1}{t^{1-\varepsilon}}.$$

\item If $N_3' \ll n_2 \ll N_3$, we have
\begin{align}\label{Integral(2)}
\mathcal{H}(n_2) \ll \frac{Q^{3/2}Rp^{\lambda/2}n_1''^{1/2}}{t^{1-\varepsilon}NX^{3/2}n_2^{1/2}}.
\end{align}
\item If $q=q'$, we have $\mathcal{H}(0) \ll t^{-A}$ unless $|m-m'|\ll t^{\varepsilon}\left(\frac{N^2X^2M}{Q^2R^2p^{2\lambda}t^2}
    +\frac{MQRp^{\lambda}}{NX}\right)$.
\end{enumerate}
 \end{lemma}
\begin{proof}
Plugging \eqref{J-stationary phase-2} into \eqref{H1(n2)}, the evaluation of $\mathcal{H}(n_2)$ is therefore reduced to estimating
\begin{align}\label{reduce}
\frac{1}{t}&\int_\mathbb{R}W(u)e\left(-\frac{n_2N_1u}
{q_1q_2p^{\lambda}n_1''}+Cu^{1/3}(g_1(D)
-g_1(D'))
+C^2u^{2/3}(g_2(D)
-g_2(D'))
+O\left(\frac{|C|^3}{t^2}\right)\right)
\mathrm{d}u,
\end{align}
where $D'=-\frac{m'N}{qp^k}$. Making a change of variable $u\to u^3$, the phase function of exponential function in the above
integral equals
$$-\frac{n_2N_1u^3}{q_1q_2p^{\lambda}n_1''}
+(g_1(D)-g_1(D'))Cu+(g_2(D)-g_2(D'))C^2u^{2}
+O\left(\frac{|C|^3}{t^2}\right).$$
Applying integration by parts, we get $\mathcal{H}(n_2)\ll t^{-A}$ if $n_2 \gg N_3$, which gives the first result in (1).
The second statement in (1) is obvious, since we use the trivial bound in \eqref{reduce}.

It is easy to see that
\begin{align}\label{CDg2}
C^2(g_2(D)-g_2(D'))\ll \frac{C(y_0^{1/3}+y_0'^{1/3})}{t}
\left|C(y_0^{1/3}-y_0'^{1/3})\right|
\ll\left|C(g_1(D)-g_1(D'))\right|t^{-\varepsilon},
\end{align}
where we have used $y_0'=\frac{t}{2\pi D'}\asymp 1$ and $C\asymp \frac{NX}{p^\lambda RQ}\ll t^{1-\varepsilon}$. Therefore, if $N_3' \ll n_2 \ll N_3$,
the $u$-integral is $O(t^{-A})$ unless $\left|C(g_1(D)-g_1(D'))\right|\asymp \frac{n_2N_1u}{qq'p^{\lambda}n_1''}$. By second derivative test, we get
\eqref{Integral(2)}.

For $n_2=0$ and $q=q'$, we may rewrite the above $u$-integral as
$$\frac{1}{t}\int_\mathbb{R}W(u^3)u^2e\left((Cg_1(D)-C'g_1(D'))u
+(C^2g_2(D)-C'^2g_2(D'))u^2+O\left(\frac{|C|^3}{t^2}\right)\right)
\mathrm{d}u.$$
Notice that $\frac{g_1(D)}{m^{1/3}}=\frac{g_1(D')}{m'^{1/3}}$ and $C(g_1(D)-g_1(D'))=\frac{Cg_1(D)}{m^{1/3}}(m^{1/3}-m'^{1/3})$. So by partial
integration and \eqref{CDg2}, the $u$-integral is $O(t^{-A})$ unless
$$m^{1/3}-m'^{1/3}\ll \left(\frac{C^3}{t^2}+1\right)\frac{M^{1/3}t^{\varepsilon}}{C}.$$
This actually proves the result in (3).
\end{proof}
\subsection{$\mathbf{\frac{NX}{p^\lambda RQ}\gg t^{1-\varepsilon}}$}
It is easy to see that $R\ll \frac{N^{1+\varepsilon}X}{p^\lambda tQ}$.
We have the following results.
\begin{lemma}\label{integral-lemma2}
Let $N_3$ be defined as in Lemma \ref{integral-lemma1}. Then, if $\frac{NX}{p^\lambda RQ}\gg t^{1-\varepsilon}$,
one has the following estimates.
\begin{enumerate}
 \item
 If $n_2 \gg N_3$, we have $\mathcal{H}(n_2) \ll t^{-A}$.
\item
If $n_2 \ll N_3$, we have $$\mathcal{H}(n_2)
\ll \frac{p^\lambda RQ}{N^{1-\varepsilon}X}.$$
    \end{enumerate}
\end{lemma}
\begin{proof}
The first result can be done by applying
integration by parts with respect to the $u$-integral.
To prove the second assertion, we observe the second
derivative of the phase function in $\mathcal{J}(N_1u, m, q)$ is
$$
\frac{t}{\pi}y^{-2}-\eta
\frac{2(N_1uN)^{1/3}}{3qp^{\lambda}}y^{-5/3}.
$$
For this to be smaller than $\frac{(N_1N)^{1/3}}{qp^{\lambda}}$ in magnitude one at least needs  $t\asymp\frac{(N_1N)^{1/3}}{qp^{\lambda}}$ and $\eta=1$. Except this case we have
$$\mathcal{J}(N_1u, m, q)\ll \sqrt{\frac{qp^{\lambda}}{(N_1N)^{1/3}}} \asymp \left(\frac{NX}{p^\lambda RQ}\right)^{-1/2} ,$$
by the second derivative bound and $N_1\asymp \frac{N^2X^3}{Q^3}$. In the special case we have
\begin{align*}
\mathcal{H}(n_2)&\ll \mathop{\int\int}\mathcal{V}(y_1)\mathcal{V}(y_2)\left|\int W(u)e\left(\frac{3(N_1uN)^{1/3}}{qp^{\lambda-r}}
(y_1^{1/3}-y_2^{1/3})\mathrm{d}u\right)\right|
\mathrm{d}y_1\mathrm{d}y_2\\
&\ll \int\int_{|y_1-y_2|\ll \frac{qp^{\lambda}}{(N_1N)^{1/3}}}\mathcal{V}(y_1)
\mathcal{V}(y_2)\mathrm{d}y_1\mathrm{d}y_2
\ll\frac{qp^{\lambda}}{(N_1N)^{1/3}}.
\end{align*}
The lemma follows.
\end{proof}

We have the following estimates for the character sum $\mathfrak{C}^*(n_2)$ ,
whose proofs we postpone to Section \ref{10}.
\begin{lemma}\label{C*{n2}}
Assume $\lambda\leq 2k/3$. Let $\lambda=2\mu+\delta$
with $\delta=0$ or $1$,  $p^\ell| n_2$ with $\ell\geq 0$.
\begin{enumerate}
\item
For $n_1''\neq1$, we have
\begin{align*}
\mathfrak{C}^*(n_2)=0.
\end{align*}
 \item
 For $n_2=0$, $\mathfrak{C}_2^*$ vanishes unless
$mq^2\equiv m'q'^2\bmod p^{\mu}$ and $q=q'$, in this case we have
\begin{align*}
\mathfrak{C}^*(n_2)\ll \widehat{q}^2(\widehat{q},m-m')p^{3\lambda}.
\end{align*}
\item
For $n_2\neq 0$, we have
\begin{align*}
\mathfrak{C}^*(n_2)\ll
\widehat{q}\widehat{q'}(\widehat{q},\widehat{q'},n_2)
p^{5\lambda/2+\min\{\ell,\mu\}+3\delta/2}.
\end{align*}
\end{enumerate}
\end{lemma}

\section{The zero frequency}
\subsection{$\mathbf{t^{\varepsilon}\ll\frac{NX}{p^\lambda RQ}\ll t^{1-\varepsilon}}$}
\ \ \\

Denote the contribution of this part to $\Omega$ by $\Omega_{0}^1$.
 By~\eqref{Omega1},~\eqref{T1} and Lemma \ref{integral-lemma1} and Lemma \ref{C*{n2}} we get
\begin{align*}
\Omega_{0}^1\ll&
\sum\limits_{\delta=0,1}\sum\limits_{q\sim R\atop n_1'|q, (q,p)=1}\frac{1}{q^2}
\underset{m\equiv m'\bmod p^{(\lambda-\delta)/2}\atop |m-m'|\ll \frac{N^2X^2M}{Q^2R^2p^{2\lambda}t^2}+\frac{MQRp^{\lambda}}{NX}}
{\sum_{m\sim M}\sum_{m'\sim M}}\underset{m\equiv ap^{k-\lambda}\bmod q\atop m\equiv a'p^{k-\lambda}\bmod q}{\;\sideset{}{^\star}\sum_{a\bmod{q}}
\;\sideset{}{^\star}\sum_{a'\bmod{q}}}
\frac{N_1}{n_1'^2}(\widehat{q},m-m')p^{2\lambda}\frac{1}{t^{1-\varepsilon}}\\
\ll
&\frac{1}{n_1'^3}\left(\frac{NX^3R}{n_1'Q^3}p^{k+2\lambda}+
\frac{N^2X^5p^{1/2}}{tRQ^5}p^{2k-\lambda/2}
+\frac{tR^2X^2p^{1/2}}{NQ^2}p^{2k+5\lambda/2}\right).
\end{align*}
Here we have used $N_1\asymp\frac{N^2X^3}{Q^3}$ and $M \ll\frac{tRp^{k}}{N}$.
The contribution of this part to $D(N, X, R,M)$ is
\begin{align*}
&\ll\frac{NX^{1/2}}{Q^{1/2}p^{k/2+\lambda}}
\bigg(\frac{NX^3R}{Q^3}p^{k+2\lambda}
+\frac{N^2X^5p^{1/2}}{tRQ^5}p^{2k-\lambda/2}
+\frac{tR^2X^2p^{1/2}}{NQ^2}p^{2k+\frac{5}{2}\lambda}\bigg)^{1/2}\\
&\ll\frac{N^{3/2}X^2R^{1/2}}{Q^{2}}+
\frac{N^2X^{3}}{Q^3R^{1/2}t^{1/2}}p^{\frac{1}{4}+\frac{k}{2}-\frac{5\lambda}{4}}
+\frac{N^{1/2}X^{3/2}Rt^{1/2}}{Q^{3/2}}p^{\frac{1}{4}+\frac{k}{2}+\frac{\lambda}{4}}.
\end{align*}

\subsection{$\mathbf{\frac{NX}{p^\lambda RQ}\gg t^{1-\varepsilon}}$}
\ \ \\

Denote the contribution of this part to $\Omega$ by $\Omega_{0}^2$.
 By~\eqref{Omega1},~\eqref{T1} and Lemmas  \ref{integral-lemma2}--\ref{C*{n2}} we get

\begin{align*}
\Omega_{0}^2\ll&
\sum\limits_{\delta=0,1}\sum\limits_{q\sim R\atop n_1'|q, (q,p)=1}\frac{1}{q^2}
\underset{m\equiv m'\bmod p^{(\lambda-\delta)/2}}
{\sum_{m\sim M}\sum_{m'\sim M}}\;\underset{m\equiv ap^{k-\lambda}\bmod q\atop m\equiv a'p^{k-\lambda}\bmod q}{\sideset{}{^\star}\sum_{a\bmod{q}}
\;\sideset{}{^\star}\sum_{a'\bmod{q}}}
\frac{N_1}{n_1'^2}(\widehat{q},m-m')p^{2\lambda}\frac{p^\lambda RQ}{N^{1-\varepsilon}X}\\
\ll&\frac{1}{n_1'^3}\left(\frac{NX^3R}{Q^3}p^{k+2\lambda}
+\frac{NX^4p^{1/2}}{Q^4}p^{2k+\lambda/2}\right).
\end{align*}
Here we have used $N_1\asymp\frac{N^2X^3}{Q^3}$ and $M \ll\frac{Xp^{k-\lambda}}{Q}$. Similar to what was said before, the contribution of this part to $D(N, X, R,M)$ is
\begin{align*}
\ll&\frac{NX^{1/2}}{Q^{1/2}p^{k/2+\lambda}}
\left(\frac{NX^3R}{Q^3}p^{k+2\lambda}
+\frac{NX^4p^{1/2}}{Q^4}p^{2k+\lambda/2}\right)^{1/2}\\
\ll&\frac{N^{3/2}X^2R^{1/2}}{Q^{2}}
+\frac{N^{3/2}X^{5/2}}{Q^{5/2}}p^{k/2-3\lambda/4+1/4}.
\end{align*}

\section{The non-zero frequencies}
\subsection{$\mathbf{t^{\varepsilon}\ll\frac{NX}{p^\lambda RQ}\ll t^{1-\varepsilon}}$}
\ \ \\

Similarly with the case of the zero frequency, denote the contribution of this case of the non-zero frequencies to $\Omega$ by
$\Omega_{\neq0}^1$. By~\eqref{Omega1},~\eqref{T1}, Lemma \ref{integral-lemma1} and Lemma \ref{C*{n2}} we have
\begin{align*}
\Omega_{\neq0}^1\ll&
\sum\limits_{\delta=0,1}
\sum_{q\sim R \atop n_1'|q, (q,p)=1}\sum_{q'\sim R \atop n_1'|q', (q',p)=1}
\frac{1}{qq'}
\sum_{m\sim M\atop(m,q)=1}\sum_{m'\sim M\atop(m',q')=1}
\underset{m\equiv ap^{k-\lambda}\bmod q
\atop m\equiv a'p^{k-\lambda}\bmod q'}
{\;\sideset{}{^\star}\sum_{a\bmod{q}}
\;\sideset{}{^\star}\sum_{a'\bmod{q'}}}\\
&\cdot\sum\limits_{0\leq\ell\leq\log q}
\frac{N_1p^{5\lambda/2+\min\{\ell,\alpha\}+3\delta/2}}
{\widehat{p}n_1'^2n_1''^2p^{\ell}}
\left(\frac{N_3'}{t}+\frac{Q^{3/2}Rp^{\lambda/2}n_1''^{1/2}}
 {tNX^{3/2}}
 \sum_{N_3'\ll n_2\ll N_3}
 \frac{(\widehat{q},\widehat{q'},n_2)}{n_2^{1/2}}\right)\\
\ll&\frac{tR^4}
{N^2n_1'^4}p^{2k+5\lambda/2+3/2}
+\frac{RX^3N}{tQ^3n_1'^4}
p^{2k-\lambda/2+3/2}
+\frac{tR^{7/2}X^{1/2}}{N^{3/2}Q^{1/2}n_1'^4}p^{2k+2\lambda+3/2}.
\end{align*}
By $N_1\asymp\frac{N^2X^3}{Q^3}$ and $M \ll\frac{tRp^{k}}{N}$. Then the contribution of this part to $D(N, X, R,M)$ is
\begin{align*}
\ll&\frac{NX^{1/2}}{Q^{1/2}p^{k/2+\lambda}}
\bigg(\frac{tR^4}
{N^2}p^{2k+5\lambda/2+3/2}
+\frac{RX^3N}{tQ^3}p^{2k-\lambda/2+3/2}
+\frac{tR^{7/2}X^{1/2}}{N^{3/2}Q^{1/2}n_1'^4}
p^{2k+2\lambda+3/2}\bigg)^{1/2}\\
\ll&\frac{t^{1/2}X^{1/2}R^2}{Q^{1/2}}p^{k/2+\lambda/4+3/4}
+\frac{N^{3/2}X^2R^{1/2}}{Q^{2}t^{1/2}}p^{k/2-5\lambda/4+3/4}
+\frac{N^{1/4}X^{3/4}R^{7/4}}{Q^{3/4}}t^{1/2}p^{k/2+3/4}.
\end{align*}

\subsection{$\mathbf{\frac{NX}{p^\lambda RQ}\gg t^{1-\varepsilon}}$}
\ \\

Denote the contribution of this part to $\Omega$ by $\Omega_{\neq0}^2$.
 By~\eqref{Omega1},~\eqref{T1} and Lemmas \ref{integral-lemma2}--\ref{C*{n2}} we get

\begin{align*}
\Omega_{\neq0}^2\ll&
\sum\limits_{\delta=0,1}
\sum_{q\sim R \atop n_1'|q, (q,p)=1}\sum_{q'\sim R \atop n_1'|q', (q',p)=1}
\frac{1}{qq'}
\sum_{m\sim M\atop(m,q)=1}\sum_{m'\sim M\atop(m',q')=1}\\
&\;\;\;\;\;\;\;\;\;\;\;\;\;\cdot\underset{m\equiv ap^{k-\lambda}\bmod q
\atop m\equiv a'p^{k-\lambda}\bmod q'}
{\;\sideset{}{^\star}\sum_{a\bmod{q}}
\;\sideset{}{^\star}\sum_{a'\bmod{q'}}}
\sum\limits_{0\leq\ell\leq\log q}
\frac{N_1p^{5\lambda/2+3/2\delta}}
{\widehat{p}n_1'^2n_1''^2}\cdot\frac{p^{\lambda}qQ}{NXp^{\ell}}N_3\\
\ll&\frac{R^2X^2}{Q^2n_1'^4n_1''^2}
p^{2k+\lambda/2+3\delta/2}.
\end{align*}
By $N_1\asymp\frac{N^2X^3}{Q^3}$ and $M \ll\frac{Xp^{k-\lambda}}{Q}$. Then the contribution of this part to $D(N, X, R,M)$ is
\begin{align*}
\ll& \frac{NX^{1/2}}{Q^{1/2}p^{k/2}+\lambda}
\left(\frac{R^2X^2}{Q^2}
p^{2k+\lambda/2+3\delta/2}\right)^{1/2}\\
\ll&\frac{NX^{3/2}R}{Q^{3/2}}p^{k/2-3\lambda/4+3/4}.
\end{align*}

\section{The non-oscillating case}\label{non-oscillating}

Now we assume that $\frac{NX}{p^{\lambda}RQ}\ll t^{\varepsilon}$ and  $\frac{n_1^2n_2N}{q^3p^{3\lambda}}\ll t^{\varepsilon}$.
In this section we allow abuse of some notations without causing ambiguity.
 Rewrite $\mathfrak{J}(m,q,x)$ as
$$
\int_{\mathbb{R}}V(y)e\left(-\frac{Nxy}{Qqp^\lambda}\right)
e\left(-\frac{t\log y}{2\pi}-\frac{mNy}{qp^k} \right)\mathrm{d}y.
$$
Let
$$
h(y)=-\frac{t\log y}{2\pi}-\frac{mNy}{qp^k}.
$$
Then we have
$$
h'(y)=-\frac{t}{2\pi y}-\frac{mN}{qp^k},\quad h''(y)=\frac{t}{2\pi y^2},
\quad h^{(j)}(y) \asymp_{j} t, \quad j\geq 2.
$$
Hence, by Lemma \ref{lemma:exponentialintegral} the integral is negligible small
unless $\frac{mN}{qp^k}\asymp t$, in which case we have the stationary
phase point $y_0=\frac{-qp^kt}{2\pi mN}$ and
$$
\mathfrak{J}(m,q,x)=\frac{1}{\sqrt{t}}
e\left(-\frac{t}{2\pi}\log\frac{-qp^kt}{2\pi emN}\right)
V_{x}\left(\frac{-qp^kt}{2\pi mN}\right)+O\left(t^{-A}\right),
$$
where $V_{x}$ is a $t^{\varepsilon}$-inert function.
Together with \eqref{s0}, we have $D(N,X,R)$ is equal to up to a negligibly
small error term
\begin{align*}
&\frac{N^{1-it}\chi(q)\tau_{\chi}}{Qp^{k}}\sum_{q\sim R}\;\sideset{}{^\star}\sum_{a\bmod qp^{\lambda}}
\sum_{\eta=\pm1}\sum_{n_1|qp^{\lambda}}\sum_{n_2}
\frac{A(n_2,n_1)}{n_1n_2}
\int_{\mathbb{R}}g(q,x)U\left(\frac{\pm x}{X}\right)\notag \\
&\cdot\sum_{m\equiv ap^{k-\lambda}\bmod q}\overline{\chi}
\left(m-ap^{k-\lambda}\right)
S\left(\overline{a},\eta n_2;\frac{qp^{\lambda}}{n_1}\right)
\left(\frac{n_1^2 n_2N}{q^3p^{3\lambda}}\right)^{1/2}\notag\\
&\cdot
\frac{1}{\sqrt{t}}
e\left(-\frac{t}{2\pi}\log\frac{-qp^kt}{2\pi emN}\right)
\Phi_{x}^{\mathrm{sgn}(\eta)}
\left(\frac{n_1^2 n_2}{q^3p^{3\lambda}}\right)
V_{x}\left(\frac{-qp^kt}{2\pi mN}\right)
\mathrm{d}x,
\end{align*}
where $\Phi_{x}^{\mathrm{sgn}(\eta)}$ is a $t^{\varepsilon}$-inert function.
Rearranging the sums, inserting a dyadic partition for $n_2$-sum and
estimating integral trivially, the above is bounded by
$$N^{\varepsilon}\sup_{1\ll N_0\ll
\frac{q^3p^{3\lambda}}{N}t^{\varepsilon}}
\sup_{x\asymp X}|\mathcal{D}(N, X, R,M)|,$$
where
\begin{align}
\mathcal{D}(N, X, R,M)=&\frac{N^{3/2}X}{t^{1/2}Qp^{k/2+3\lambda/2}}
\sum_{\eta=\pm1}\sum_{q\sim R\atop(q,p)=1}\frac{1}{q^{3/2}}
\sum_{n_1|qp^{\lambda}}
\sum_{n_2\asymp\frac{N_0}{n_1^2}}
\frac{A(n_2,n_1)}{n_2^{1/2}}\notag \\
&\cdot\sum_{m\asymp \frac{tRp^k}{N}} \;\sideset{}{^\star}\sum_{a\bmod{qp^\lambda}\atop m\equiv ap^{k-\lambda}\bmod q}\overline{\chi}
\left(m-ap^{k-\lambda}\right)
S\left(\overline{a},\eta n_2;\frac{qp^{\lambda}}{n_1}\right)\notag\\
&\cdot \Phi_{x}^{\mathrm{sgn}(\eta)}
\left(\frac{n_1^2 n_2}{q^3p^{3\lambda}}\right)
V_{x}\left(\frac{-qp^kt}{2\pi mN}\right),
\end{align}
Now we use the Cauchy-Schwartz inequality
and \eqref{theta3} with $n_1=n_1'n_1''$, $n_1'|q, \; n_1''|p^\lambda$ to get
$$
\mathcal{D}(N, X, R,M)\ll\frac{N^{3/2}X}{t^{1/2}Qp^{k/2+3\lambda/2}}
\sum\limits_{\eta=\pm1}
\sum\limits_{(n_{1}',p)=1}\sum\limits_{n_{1}''|p^{\lambda}}
(n_{1}'n_{1}'')^{\theta_3}\;\Omega^{1/2},
$$
where
\begin{align*}
\Omega=\sum\limits_{n_2}W\left(\frac{n_{1}'^2n_{1}''^2n_2}{N_0}\right)&
\bigg|\sum_{q\sim R \atop n_1'|q, (q,p)=1}\frac{1}{q^{3/2}}\sum_{m\asymp \frac{tRp^k}{N}\atop (m,q)=1}\;\sideset{}{^\star}\sum_{a\bmod{q}\atop m\equiv ap^{k-\lambda}\bmod q}\\ &\mathfrak{C}(m,a,q,n_1',n_1'',n_2)\Phi_{x}^{\mathrm{sgn}(\eta)}
\left(\frac{n_{1}'^2n_{1}''^2n_2}{q^3p^{3\lambda}}\right)
V_{x}\left(\frac{-qp^kt}{2\pi mN}\right)\bigg|^2,
\end{align*}
with
$\mathfrak{C}(m,a,q,n_1',n_1'',n_2)$ is defined in \eqref{character}.
Opening the square we get
\begin{align}\label{Omega}
\Omega\ll
\sum_{q\sim R \atop n_1'|q, (q,p)=1}
\sum_{q'\sim R \atop n_1'|q', (q',p)=1}\frac{1}{(qq')^{3/2}}
\sum_{m\asymp \frac{tRp^k}{N}\atop (m,q)=1}
\sum_{m'\asymp \frac{tRp^k}{N}\atop (m',q')=1}
\underset{m\equiv ap^{k-\lambda}\bmod q \atop m'\equiv a'p^{k-\lambda}\bmod q'}{\;\sideset{}{^\star}\sum_{a\bmod{q}}
\;\sideset{}{^\star}\sum_{a'\bmod{q'}}}\mathcal{T},
\end{align}
where
\begin{align*}
\mathcal{T}=\sum\limits_{n_2\geq1}\Phi\left(\frac{n_{1}'^2n_{1}''^2n_2}{N_0}\right)
\mathfrak{C}(m,a,q,n_1',n_1'',n_2)
\overline{\mathfrak{C}(m',a',q',n_1',n_1'',n_2)},
\end{align*}
with
$\Phi\left(\frac{n_{1}'^2n_{1}''^2n_2}{N_0}\right)$ is a smooth
compactly supported function which contains the weight function
$\Phi_{x}^{\mathrm{sgn}(\eta)}
\left(\frac{n_{1}'^2n_{1}''^2n_2}{q^3p^{3\lambda}}\right)
\overline{\Phi_{x}^{\mathrm{sgn}(\eta)}
\left(\frac{n_{1}'^2n_{1}''^2n_2}{q^3p^{3\lambda}}\right)}$.

Note that in the (3) of Lemma~\ref{INT}, by taking $\sigma=1/2$ and
making a change of variable, we can get
$$
\Psi_{x}^{\pm}(z)=(zN)^{1/2}\int_{\mathbb{R}}(\pi^3zN)^{-i\tau}
\gamma^{\pm}_{3}(1/2+i\tau)\int_0^{\infty}W(y)
e\left(\frac{xNy}{p^\lambda qQ}\right)y^{-1/2-i\tau}\mathrm{d}y\mathrm{d}\tau.
$$
By repeated integration by parts for the $y$-integral, we can truncate $\tau$
at $\tau\ll t^{\varepsilon}$ with a negligibly small error term and get
$$
\Psi_{x}^{\pm}(z)=(zN)^{1/2}\Phi_{x}^{\pm}(z)+O(t^{-A}),
$$
where
$$
\Phi_{x}^{\pm}(z)=\frac{1}{2\pi^{5/2}}\int_{|\tau|\leq t^{\varepsilon}}
(\pi^3zN)^{-i\tau}\gamma^{\pm}_{3}(1/2+i\tau)\int_0^{\infty}W(y)
e\left(\frac{xNy}{p^\lambda qQ}\right)y^{-1/2-i\tau}\mathrm{d}y\mathrm{d}\tau.
$$
The function $\Phi_{x}^{\pm}(z)$ satisfies
$$
\frac{\partial^{j}}{\partial z^j}
\Phi_{x}^{\pm}(z)
\ll_{j} t^{\varepsilon}z^{-j},
$$
then we derive that
$$
\frac{\partial^{j}}{\partial n_2^j}
\Phi\left(\frac{n_{1}'^2n_{1}''^2n_2}{N_0}\right)
\ll_{j} t^{\varepsilon}n_2^{-j},\ \ \ \ j\geq0.
$$
By the Poisson summation formula modulo
$\widehat{q}\widehat{q'}\widehat{p}$, we arrive at
\begin{align}\label{T}
\mathcal{T}=\frac{N_0}{\widehat{q}\widehat{q'}\widehat{p}n_1'^2n_1''^2}
\sum\limits_{n_2\in \mathbb{Z}}\left|\mathfrak{C}^*(n_2)\right|\cdot
\left|\mathcal{H}(n_2)\right|,
\end{align}
where $\mathfrak{C}^*(n_2)$ is defined as in \eqref{Cn2} and
$$
\mathcal{H}(n_2)=\int_{\mathbb{R}}\Phi(u)
e\left(-\frac{n_2N_0u}{qq'p^{\lambda}n_1''}\right)\mathrm{d}u.
$$
We can get an upper bound of $\mathcal{H}(n_2)$ by repeated integration by parts,
that is
\begin{align}\label{H(n2)}
\mathcal{H}(n_2)\ll\left\{
         \begin{array}{ll}
           t^{-A}, & {if\ n_2\gg\frac{R^2p^{\lambda}n_1''}{N_0}t^{\varepsilon};} \\
          \\
           t^{\varepsilon}, & {if\ n_2\ll\frac{R^2p^{\lambda}n_1''}{N_0}t^{\varepsilon}.}
         \end{array}
       \right.
\end{align}

\subsection{The zero frequency}
Denote the contribution of this part to $\Omega$ by $\Omega_0$.
By \eqref{Omega}, \eqref{H(n2)} and Lemma~\ref{C*{n2}} we get
\begin{align*}
\Omega_0\ll&\sum\limits_{\delta=0,1}
\sum\limits_{q\sim R\atop n_1'|q, (q,p)=1}\frac{1}{q^3}
\underset{m\equiv m' \bmod p^{(\lambda-\delta)/2}}{\sum\limits_{m\asymp \frac{tRp^k}{N}\atop (m,q)=1}\sum\limits_{m'\asymp \frac{tRp^k}{N}\atop (m',q')=1}}
\underset{m\equiv ap^{k-\lambda}\bmod q \atop m'\equiv a'p^{k-\lambda}\bmod q'}{\;\sideset{}{^\star}\sum_{a\bmod{q}}
\;\sideset{}{^\star}\sum_{a'\bmod{q'}}}
\frac{N_0t^{\varepsilon}}{n_1'^2}(\widehat{q},m-m')p^{2\lambda}\\
\ll&\frac{1}{n_1'^3}\left(\frac{tR^3}{N^2n_1'}p^{k+5\lambda}+
\frac{t^2R^3}{N^3}p^{2k+9\lambda/2+1/2}
\right).
\end{align*}
The contribution of this part to $\mathcal{D}(N, X, R,M)$ is
$$
\ll \frac{N^{\frac{1}{2}}XR^{3/2}}{Q}p^{\lambda}
+\frac{t^{\frac{1}{2}}XR^{3/2}}{Q}p^{\frac{k}{2}+\frac{3}{4}\lambda+\frac{1}{4}}.
$$

\subsection{The non-zero frequencies}
Denote the contribution of this case of the non-zero frequencies to $\Omega$
by $\Omega_{\neq0}$,
\begin{align*}
\Omega_{\neq0}\ll&
\sum_{\delta=0,1}\sum\limits_{q\sim R\atop n_1'|q, (q,p)=1}
\sum\limits_{q'\sim R\atop n_1'|q', (q',p)=1}\frac{1}{(qq')^{3/2}}
\sum\limits_{m\asymp \frac{tRp^k}{N}\atop (m,q)=1}\sum\limits_{m'\asymp \frac{tRp^k}{N}\atop (m',q')=1}
\underset{m\equiv ap^{k-\lambda}\bmod q \atop m'\equiv a'p^{k-\lambda}\bmod q'}{\;\sideset{}{^\star}\sum_{a\bmod{q}}
\;\sideset{}{^\star}\sum_{a'\bmod{q'}}}\\
&\cdot\sum_{n_2\ll\frac{R^2p^{\lambda}n_1''}{N_0}t^{\varepsilon}}
\sum_{0\leq\ell\leq\log q}
\frac{N_0t^{\varepsilon}}{\widehat{p}n_1'^2p^{\ell}}
(\widehat{q},\widehat{q'},n_2)
p^{5\lambda/2+\min\{\ell,\alpha\}+3\delta/2}\\
\ll&\frac{t^{2+\varepsilon}R^3}{N^2n_1'^4}p^{2k+5\lambda/2+3/2}.
\end{align*}
Then the contribution of this part to $\mathcal{D}(N, X, R,M)$ is
\begin{align*}
\ll& \frac{N^{1/2}XR^{3/2}}{Q}t^{1/2}p^{k/2-\lambda/4+3/4}\\
\ll& \frac{R^{5/2}t^{1/2}}{N^{1/2}}p^{k/2+3\lambda/4+3/4}.
\end{align*}
Here we used the condition $X\ll\frac{p^{\lambda}RQ}{N}t^{\varepsilon}$.
So we have
$$
\mathcal{D}(N, X, R,M)\ll \frac{N^{\frac{1}{2}}XR^{3/2}}{Q}p^{\lambda}
+\frac{t^{\frac{1}{2}}XR^{3/2}}{Q}p^{\frac{k}{2}+\frac{3}{4}\lambda+\frac{1}{4}}
+ \frac{R^{5/2}t^{1/2}}{N^{1/2}}p^{k/2+3\lambda/4+3/4}.
$$
\section{Conclusion}

Now we ready to give a proof of proposition~\ref{upper bound}.
Recall that in Aggarwal \cite{Agg2} and Sun-Zhao \cite{SZ},
the authors took $Q$ to be $N^{1/2}/t^{1/5}$
and $N^{1/2}/p^{\lambda/2}$, respectively. This motivates us
to choose $Q=N^{1/2}/(p^{\lambda/2}t^{1/5})$. As we will see, after balancing finally,
which coincides with Aggarwal \cite{Agg2} and Sun-Zhao \cite{SZ}.
Firstly, by taking $Q=N^{1/2}/(p^{\lambda/2}T)$,
we get
\begin{align*}
S(N)\ll\;&N^{3/4}T^{3/4}p^{3\lambda/4}
+N^{1/2}T^{3/2}t^{-1/2}p^{k/2+\lambda/4+1/4}+N^{1/4}T^{1/4}t^{1/2}p^{k/2+\lambda/2+1/4}\\
&+N^{3/4}T^{3/4}p^{3\lambda/4+1/4}
+N^{1/4}T^{5/4}p^{k/2+\lambda/2+1/4}
+t^{1/2}N^{3/4}T^{-3/4}p^{k/2-\lambda/2+3/4}\\
&+t^{-1/2}N^{3/4}T^{3/4}p^{k/2-\lambda/2+3/4}
+t^{1/2}N^{3/4}T^{-1/2}p^{k/2-\lambda/2+3/4}+N^{3/4}T^{1/4}p^{k/2-\lambda/2+3/4}.
\end{align*}
Here we used $R\ll Q$ and $X\ll N^{\varepsilon}$.
Then we take $\lambda=\lfloor2k/5\rfloor+1$ and $T=t^{2/5}$.
We conclude that
$$
S(N)\ll p^{3/4}N^{1/2+\varepsilon}(\mathfrak{q}t)^{3/4-3/40},
$$
provided that $N\ll(\mathfrak{q}t)^{3/2+\varepsilon}$.

\section{Character sums}\label{10}
In this section we estimate the character sums in \eqref{Cn2}
\begin{align*}
\mathfrak{C}^*(n_2)=\sum\limits_{\beta\bmod\widehat{q}\widehat{q'}\widehat{p}}
\mathfrak{C}(m,a,q,n_1',n_1'',\beta)
\overline{\mathfrak{C}(m',a',q',n_1',n_1'',\beta)}
e\left(\frac{\eta n_2\beta}{\widehat{q}\widehat{q'}\widehat{p}}\right).
\end{align*}

Write $\beta=\widehat{q}\widehat{q'}
\overline{\widehat{q}\widehat{q'}}b_{1}+\widehat{p}
\overline{\widehat{p}}b_{2}$, with $b_1\bmod\widehat{p}$,
$b_2\bmod \widehat{q}\widehat{q'}$.
We obtain
$$
\mathfrak{C}^*(n_2)=\mathfrak{C}_1^*\mathfrak{C}_2^*,
$$
where
\begin{align*}
\mathfrak{C}_1^*=\sum_{b\bmod \widehat{q}\widehat{q'}}
   S\left(a\overline{\widehat{p}},b\overline{\widehat{p}};\widehat{q}\right)
   S\left(a'\overline{\widehat{p}},b\overline{\widehat{p}};\widehat{q'}\right)
   e\left(\frac{n_2\overline{\widehat{p}}b}{\widehat{q}\widehat{q'}}\right),
\end{align*}
and
\begin{align*}
\mathfrak{C}_2^*=&\sum_{b\bmod\widehat{p}}\
\sideset{}{^\star}\sum_{c_1\bmod p^{\lambda}}
   \overline{\chi}\left(m-c_1p^{k-\lambda}\right)
   S\left(\overline{c_1}\overline{\widehat{q}},
 b\overline{\widehat{q}};\widehat{p}\right)\\
& \cdot\sideset{}{^\star}\sum_{c_2\bmod p^{\lambda}}
   \chi\left(m'-c_2p^{k-\lambda}\right)
   S\left(\overline{c_2}\overline{\widehat{q'}},b
   \overline{\widehat{q'}};\widehat{p}\right)
   e\left(\frac{\overline{\widehat{q}\widehat{q'}}bn_2}{\widehat{p}}\right).
\end{align*}
The following estimate for the character sum $\mathfrak{C}_1^*$ was proved in \cite{Mun3}.
\begin{lemma}
We have
\begin{align*}
\mathfrak{C}_1^* \ll \widehat{q}\widehat{q'}(\widehat{q},\widehat{q'},n_2).
\end{align*}
Moreover, for $n_2=0$, the character sums vanish unless $q=q'$ in which case
$$
\mathfrak{C}_1^*\ll \widehat{q}^2(\widehat{q},m-m').
$$
\end{lemma}
To estimate the character sum $\mathfrak{C}_2^*$, we use the
 strategy in \cite{Mun3} and \cite{SZ} to prove the following results.
\begin{lemma}
Assume $\lambda\leq 2k/3$. Let $\lambda=2\mu+\delta$
with $\delta=0$ or $1$,  $p^\ell| n_2$ with $\ell\geq 0$.
\begin{enumerate}
\item
For $n_1''\neq1$, we have
\begin{align*}
\mathfrak{C}_2^*=0.
\end{align*}
 \item
 For $n_2=0$, $\mathfrak{C}_2^*$ vanishes unless
$mq^2\equiv m'q'^2\bmod p^{\mu}$, in this case we have
\begin{align}\label{10.1}
\mathfrak{C}_2^*\ll p^{3\lambda}.
\end{align}
\item
For $n_2\neq 0$, we have
\begin{align}\label{10.2}
\mathfrak{C}_2^*\ll
p^{5\lambda/2+\min\{\ell,\mu\}+3\delta/2}.
\end{align}
\end{enumerate}
\end{lemma}

\begin{proof}
Opening the Kloosterman sums and executing the sum over $b$, we arrive at
\begin{align}\label{openK}
\mathfrak{C}_2^*=&\,\widehat{p}
\sideset{}{^\star}\sum_{c_1\bmod  p^{\lambda}}
\;\sideset{}{^\star}\sum_{c_2\bmod  p^{\lambda}}
\overline{\chi}\left(m-c_1p^{k-\lambda}\right)
\chi\left(m'-c_2p^{k-\lambda}\right)\nonumber\\
 &\cdot\underset{\alpha \widehat{q}+\alpha'\widehat{q'}+\widehat{q}\widehat{q'}n_2\equiv0 \bmod \widehat{p}}
{\sideset{}{^\star}\sum_{\alpha\bmod\widehat{p}}
\;\sideset{}{^\star}\sum_{\alpha'\bmod\widehat{p}}}
e\left(\frac{\alpha\overline{c_1\widehat{q}}+\alpha'\overline{c_2\widehat{q'}}}
{\widehat{p}}\right).
\end{align}
For $n_2 \equiv 0 \bmod \widehat{p}$, we get $\alpha' \equiv -\alpha\widehat{q}\overline{\widehat{q'}} \bmod \widehat{p}$ and
it follows that
\begin{align*}
\mathfrak{C}_2^*=&\,\widehat{p}^2
\sideset{}{^\star}\sum_{c\bmod p^{\lambda}}
\overline{\chi}\left(m-cp^{k-\lambda}\right)
\chi\left(m'-c(\widehat{q}\overline{\widehat{q'}})^2p^{k-\lambda}\right)\nonumber\\
 &-\widehat{p}
\sideset{}{^\star}\sum_{c_1\bmod p^{\lambda}}
\;\sideset{}{^\star}\sum_{c_2\bmod p^{\lambda}}
\overline{\chi}\left(m-c_1p^{k-\lambda}\right)
\chi\left(m'-c_2p^{k-\lambda}\right).
\end{align*}
The last double sum is clearly bounded by $O(\widehat{p}p^{2\lambda})$. The other sum has no cancelation if
$m\widehat{q}^2 \equiv m'\widehat{q'}^2 \bmod p^\lambda$
and we get
\begin{align}\label{10.4}
\mathfrak{C}_2^*\ll \widehat{p}p^{2\lambda}.
\end{align}

Write
$p^{\lambda}=p^{2\nu_1+\delta_1}$, $\widehat{p}=p^{\lambda}/n_1''=p^{2\nu_2+\delta_2}$,
$\delta_1=0$ or 1, $\delta_2=0 \text{\ or\ } 1$,  $\nu_1\geq 1$.
Write $c_1=b_1p^{\nu_1+\delta_1}+b_2$, $c_2=h_1p^{\nu_1+\delta_1}+h_2$,
where $b_1$, $h_1$ vary over a set of representatives
of the residue classes modulo $p^{\nu_1}$ respectively,
and $b_2$, $h_2$ vary over a set of representatives
of residue classes modulo $ p^{\nu_1+\delta_1}$
prime to $ p^{\nu_1+\delta_1}$ respectively.

If $\nu_2=0$, $n_1''=p^{\lambda-1}\text{\;or\;}p^\lambda$,  we have $\widehat{p}=p\text{\ or\;}1 $. In the former case, we get
\begin{align*}
\mathfrak{C}_2^*
=&p\sideset{}{^\star}\sum_{b_2\bmod p^{\nu_1+\delta}}\;
\sideset{}{^\star}\sum_{h_2\bmod p^{\nu_1+\delta}}
\overline{\chi}\left(m-b_2p^{k-2\nu_1-\delta}\right)
\chi\left(m'-h_2p^{k-2\nu_1-\delta}\right)
\nonumber\\
   &\cdot\sideset{}{^\star}\sum_{\alpha\bmod p}
   e\left(\frac{\overline{\widehat{q'}h_2}
   -\widehat{q'}
   \overline{\widehat{q}\left(\widehat{q}+ n_2\alpha\right)
   b_2}}{p}\alpha\right)
\sum_{b_1 \bmod p^{\mu}}
\chi\left(1+\overline{m-b_2
   p^{k-2\nu_1-\delta}} p^{k-\nu_1}b_1\right)\nonumber\\
   &\cdot\sum_{h_1 \bmod p^{\nu_1}}
 \chi\left(1-\overline{m'-h_2p^{k-2\nu_1}} p^{k-\nu_1}h_1\right).
\end{align*}
Recall $\chi$ is a primitive character of modulus $p^{k}$ and $k>\lambda\geq 2\nu_1$.
Thus $\chi(1+zp^{k-\nu_1})$
is an additive character to modulus $p^{\nu_1}$, so there exists an integer
$\xi$ (uniquely determined modulo $p^{\nu_1}$), $(\xi,p)=1$, such that
$\chi(1+zp^{k-\nu_1})=\exp(2\pi i \xi z/p^{\nu_1})$. Therefore,
$\mathfrak{C}_2^*=0$. For $n_1''=p^{\lambda}$, similar conclusion is easier to prove.

Now we assume $\nu_2\geq 1$, write $\alpha=\alpha_1p^{\nu_2+\delta_2}+\alpha_2$,
where $\alpha_1$ runs over a set of representatives
of the residue classes modulo $p^{\nu_1}$,
and $\alpha_2$ runs over a set of representatives
of residue classes modulo $ p^{\nu_2+\delta_2}$
prime to $ p^{\nu_2+\delta_2}$.
Then by \eqref{openK}, we have
\begin{align*}
\mathfrak{C}_2^*
=&p^{2\nu_2+\delta_2}
\sum_{b_2\bmod p^{\nu_1+\delta_1}}
\sum_{h_2\bmod p^{\nu_1+\delta_1}}\,
\sideset{}{^\star}
\sum_{\alpha_2 \bmod p^{\nu_2+\delta_2}}
\sum_{b_1\bmod p^{\nu_1}}
\sum_{h_1\bmod p^{\nu_1}}\\
&\cdot\sum_{\alpha_1 \bmod p^{\nu_2}}
\overline{\chi}\left(m-(b_2+b_1p^{\nu_1+\delta_1})
   p^{k-2\nu_1-\delta_1}\right)
   \chi\left(m'-(h_2+h_1p^{\nu_1+\delta_1})
   p^{k-2\nu_1-\delta_1}\right)\\
&\cdot e\left(\frac{
   \overline{\widehat{q'}(h_2+h_1p^{\nu_1+\delta_1})}-\widehat{q'}
   \overline{\widehat{q}\left(\widehat{q}+ n_2\alpha_2+n_2\alpha_1p^{\nu_2+\delta_2}\right)
   (b_2+b_1p^{\nu_1+\delta_1})}}{p^{2\nu_2+\delta_2}}
   (\alpha_2+\alpha_1p^{\nu_2+\delta_2})\right).
\end{align*}
Note that $k>\lambda= 2\nu_1+\delta_1\geq 2\nu_2+\delta_2$ and
$\overline{a+bp^{\nu_1}}\equiv \overline{a}(1-\overline{a}bp^{\nu_1})\bmod p^{2\nu_1}$. Thus
\begin{align}\label{c1c2}
\mathfrak{C}_2^*
=p^{2\nu_1+3\nu_2+\delta_2} \mathop{\sum_{b_2(\text{{\rm mod }} p^{\nu_1+\delta_1})}
\sum_{h_2\bmod p^{\nu_1+\delta_1}}\,
 \sideset{}{^\star}\sum_{\alpha_2\bmod p^{\nu_2+\delta_2}}}_{
 \overline{(\widehat{q_1}+n_2\alpha_2)^2b_2}
    \widehat{q_2}^2n_2\alpha_2- \overline{
   (\widehat{q_1}+n_2\alpha_2)b_2}\widehat{q_2}^2
   +\widehat{q_1}\overline{h_2}\equiv 0\bmod p^{\nu_2} }f(b_2,h_2,\alpha_2)\mathcal {C}_1\mathcal {C}_2,
\end{align}
where
\begin{align*}
\mathcal {C}_1&=\frac{1}{p^{\nu_1}}\sum_{b_1 \bmod p^{\nu_1}}
   \chi\left(1+\overline{m_{1}-b_2
   p^{k-2\nu_1-\delta_1}}p^{k-\nu_1}b_{1}\right)
   e\left(\frac{\widehat{q_2}
   \overline{\widehat{q_1}(\widehat{q_1}+n_2\alpha_2)b_2^2}
                \alpha_2n_1''}{p^{\nu_1}}b_1\right),\\
\mathcal {C}_2&=\frac{1}{p^{\nu_1}}\sum_{h_1 \bmod p^{\nu_1}}
   \chi\left(1-\overline{m'-h_2p^{k-2\nu_1-\delta_1}}p^{k-\nu_1}h_1\right)
   e\left(\frac{-\overline{\widehat{q_2}h_2^2} \alpha_2 n_1''}
   {p^{\nu_1}}h_1\right),
\end{align*}
and
\begin{align*}
f(b_2,h_2,\alpha_2)=\overline{\chi}\left(m-b_2p^{k-2\nu_1-\delta_1}\right)
   \chi\left(m'-h_2p^{k-2\nu_1-\delta_1}\right)
  e\left(\frac{\overline{\widehat{q'}h_2}
  -\widehat{q'}\overline{\widehat{q}(\widehat{q}+n_2\alpha_2)b_2
  }}{p^{2\nu_2+\delta_2}}d_2\right).
\end{align*}
Since $\chi(1+zp^{k-\nu_1})=\exp(2\pi i \xi z/p^{\nu_1})$ with $(\xi,p)=1$, we have
\begin{align*}
\mathcal {C}_1&=\frac{1}{p^{\nu_1}}\sum_{b_1\bmod  p^{\nu_1}}
  e\left(\frac{\overline{m-b_2
   p^{k-2\nu_1-\delta_1}}\xi}{p^{\nu_1}}b_1\right)
   e\left(\frac{\widehat{q'}\overline{\widehat{q}(\widehat{q}+n_2\alpha_2)b_2^2}
                \alpha_2n_1''}{p^{\nu_1}}b_1\right)\\
&=\delta\big(\overline{m-b_2
   p^{k-2\nu_1-\delta_1}} \xi+\widehat{q'}\overline{\widehat{q}(\widehat{q}+n_2\alpha_2)b_2^2}
                \alpha_2n_1''\equiv 0 \bmod p^{\nu_1}\big).
\end{align*}
Thus $\mathcal {C}_1$ vanishes unless $n_1''=1$ which in turn implies that $\nu_1=\nu_2$ and $\delta_1=\delta_2$.
Then we get (1).

By taking $\lambda\leq 2k/3$, we have $k\geq 3\nu_1+2\delta_1$. Hence
$\mathcal {C}_1$ vanishes unless $\overline{m} \xi+\widehat{q'}
\overline{\widehat{q_1}(\widehat{q}+n_2\alpha_2)b_2^2}
                \alpha_2\equiv 0\bmod p^{\nu_1}$.
Similarly,
\begin{align*}
\mathcal {C}_2=\delta\big(\overline{m'} \xi+\overline{\widehat{q'}h_2^2} \alpha_2
   \equiv 0 \bmod p^{\nu_1}\big).
\end{align*}
Plugging these into \eqref{c1c2} we obtain
\begin{align}\label{ff}
\mathfrak{C}_2^*
&=p^{5\nu_1+\delta_1} \mathop{\sum_{b_2\bmod p^{\nu_1+\delta_1}}
\sum_{h_2(\text{{\rm mod }} p^{\nu_1+\delta_1})}\,
 \sideset{}{^\star}\sum_{\alpha_2\bmod p^{\nu_1+\delta_1}}}_{\substack{
 \overline{(\widehat{q}+n_2\alpha_2)^2b_2}
    \widehat{q'}^2n_2\alpha_2- \overline{
   (\widehat{q}+n_2\alpha_2)b_2}\widehat{q_2}^2
   +\widehat{q}\overline{h_2}\equiv 0\bmod p^{\nu_1}\\
   \overline{m} \nu_1+\widehat{q_2}\overline{\widehat{q}(\widehat{q_1}+n_2\alpha_2)b_2^2}
                \alpha_2\equiv 0 \bmod p^{\nu_1}\\
   \overline{m'} \nu_1+\overline{\widehat{q'}h_2^2} \alpha_2
   \equiv 0 \bmod p^{\nu_1}
   }}f(b_2,h_2,\alpha_2).
\end{align}
To count the numbers of $b_2,h_2$ and $\alpha_2$, we solve the three congruence
equations in \eqref{ff}.

(i) If $n_2=0$ or $n_2=p^\ell n_2'$ with
$(n_2',p)=1$ and $p^\ell\geq p^{\nu_1}$, we have
\begin{align*}
\left\{\begin{array}{l}
h_2\equiv
\overline{\widehat{q'}^2}\widehat{q}^2b_2
  \bmod p^{\nu_1},\\
   \alpha_2\equiv -\overline{m} \xi\widehat{q}^2\overline{\widehat{q'}}
                b_2^2 \bmod p^{\nu_1},\\
   \alpha_2\equiv -\overline{m'} \xi\widehat{q'}h_2^2
   \bmod p^{\nu_1}.
\end{array}\right.
\end{align*}
By the last two equations, one sees that $\mathfrak{C}_2^*$ vanishes unless
$m\widehat{q}^2\equiv m'\widehat{q'}^2\bmod p^{\nu_1}$.
By \eqref{10.4}, the bound in \eqref{10.1} follows.

Moreover, for fixed $b_2$, $h_2$ and $\alpha_2$ are uniquely determined
modulo $p^{\nu_1}$. Therefore,
\begin{align}\label{6.7}
\mathfrak{C}_2^*\ll p^{6\nu_1+4\delta_1}
\ll p^{3\lambda+\delta_1}.
\end{align}

(ii) If $n_2\neq 0$, we let $n_2=p^\ell n_2'$ with $(n_2',p)=1$ and $p^\ell< p^{\nu_1}$, and let $\gamma=\overline{\widehat{q}+n_2\alpha_2}$. Then
$\alpha_2\equiv\overline{n_2'}
(\overline{\gamma}-\widehat{q})/p^\ell \bmod  p^{\nu_1-\ell}$ and
the three equations give
\begin{align}\label{6.8}
\left\{\begin{array}{l}
b_2\equiv \widehat{q'}^2\gamma^2h_2\bmod p^{\nu_1},\\
\gamma\equiv\overline{\widehat{q}}\left(1+\overline{m}\xi\widehat{q}
\overline{\widehat{q'}}
n_2b_2^2\right)\bmod p^{\nu_1},\\
\overline{\gamma}\equiv
\widehat{q}\left(1-\overline{m'}\xi\overline{\widehat{q}}
\widehat{q'}
n_2h_2^2\right)\bmod p^{\nu_1}.
\end{array}\right.
\end{align}
Plugging the second equation into the first equation in \eqref{6.8} we get
\begin{align*}
b_2\equiv \widehat{q'}^2\overline{\widehat{q}}^2
\left(1+\overline{m}\xi\widehat{q}\overline{\widehat{q'}}
n_2b_2^2\right)^2h_2\bmod p^{\nu_1}.
\end{align*}
By the above equation and the last two equations in \eqref{6.8} we get
\begin{align}\label{6.10}
& \left(\overline{m}\xi\widehat{q}\overline{\widehat{q'}}\right)^5u^5
   +4\left(\overline{m}\xi\widehat{q}\overline{\widehat{q'}}\right)^4u^4
   +6\left(\overline{m}\xi\widehat{q}\overline{\widehat{q'}}\right)^3u^3
   +4\left(\overline{m}\xi\widehat{q}\overline{\widehat{q'}}\right)^2u^2\nonumber\\
& \qquad-\overline{mm'}\xi^2\widehat{q}^4\overline{\widehat{q_2}}^4u^2
   +\overline{m}\xi\widehat{q}\overline{\widehat{q'}}u
   -\overline{m'}\xi\widehat{q}^3\overline{\widehat{q'}}^3u
   \equiv 0 \bmod p^{\nu_1},
\end{align}
where $u=n_2b_2^2$. Thus there are at most 5 roots modulo $p^{\nu_1}$ for $u$.
Therefore, there are at most 10 roots modulo $p^{\nu_1-\ell}$ for $b_2$.
For fixed $u$,
$\gamma$ is uniquely determined modulo $p^{\nu_1}$ and for fixed $\gamma$ and $b_2$,
$h_2$ is uniquely determined modulo $p^{\nu_1}$ by the first equation in \eqref{6.8}.
Then by the last congruence equation in
\eqref{ff}, $d_2$ is uniquely determined modulo $p^{\nu_1}$.
Therefore,
\begin{align*}
\mathfrak{C}_2^*\ll p^{5\nu_1+\ell+4\delta_1}\ll p^{5\nu_1/2+\ell+3\delta_1/2}.
\end{align*}
By \eqref{6.7} and \eqref{6.10}, the bound in \eqref{10.2} follows.
\end{proof}

\bigskip
\noindent{\bf Acknowledgements.}
We are grateful to Professor Jianya Liu  for his guidance.
We gratefully acknowledge the many helpful suggestions
of Bingrong Huang, Qingfeng Sun and Zhao Xu
during the preparation of the paper.


	\end{document}